\documentclass{amsart}
\usepackage{preamble}
\usepackage[backend=biber, style=alphabetic]{biblatex}
\begin{document}
\title{On the Deformation Theory of $\bb{E}_{\infty}$-Coalgebras}
\author{Florian Riedel}
\begin{abstract}
  We introduce a notion of formally \'etale $\bb{E}_{\infty}$-coalgebras and show that they
  admit essentially unique, functorial lifts along square zero extensions of $\bb{E}_{\infty}$-rings.
  Using this, we show that for a perfect $\F_p$-algebra $k$, Weil restriction along the augmentation
  $\bb{W}(k)\to k$ induces a fully faithful functor from formally \'etale, connective
  $\bb{E}_{\infty}$-coalgebras in $k$-modules to connective $\bb{E}_{\infty}$-coalgebras
  in $p$-complete modules over the spherical Witt vectors $\bb{W}(k)$. Finally, we prove that
  for any connected space $X$, the $k$-homology $k[X]$ is a formally \'etale $\bb{E}_\infty$-coalgebra
  in $k$-modules.
  This shows that $\bb{W}(k)[X]^{\wedge}_p$ can be recovered as the essentially
  unique lift of $k[X]$ to a connective coalgebra in $p$-complete $\bb{W}(k)$-modules.
\end{abstract}

\date{\today}
\maketitle
\setcounter{tocdepth}{2}
\tableofcontents
\section{Introduction}
In this paper, we systematically study deformation problems of $\bb{E}_\infty$-coalgebras. 
The key insight is that there is a well behaved class of coalgebras for which such deformation
problems are uniquely solvable, which we call \textit{formally \'etale} coalgebras. 
Moreover, we show that there is a Witt vector style functor which is defined for all coalgebras
which produces unique lifts when restricted to formally \'etale coalgebras.

\addtocontents{toc}{\protect\setcounter{tocdepth}{1}}
\subsection{Overview}

Let $X$ be a space and throughout fix a prime $p$. We want to compare two different invariants
associated to $X$, namely the  $\F_p$-homology $\F_p[X]$ and the $p$-completed suspension
spectrum $\Sigma^\infty_+X^\wedge_p = \S[X]^\wedge_p$. The $\F_p$-chains can be recovered from
the spherical chains via the base change
\[ \S[X]^\wedge_p \otimes \F_p \simeq \F_p[X].\]
However, on the level of spectra there is in general no canonical way to go back.
For example, since any module over a field is free, we have that 
$\F_2[\bb{R}\rm{P}^2] \simeq \F_2 \oplus \Sigma^2\F_2$, but the suspension spectrum
$\S[\bb{R}\rm{P}^2]^\wedge_2 \simeq (\Sigma\S/2)^\wedge_2$ is \textit{not} free as module over 
$\S^\wedge_2$. Thus, we have two different lifts of $\F_2[\bb{R}\rm{P}^2]$ to $\S_2^\wedge$,
namely $\S[\bb{R}\rm{P}^2]^\wedge_2$s and $\S_2^\wedge \oplus \Sigma^2\S_2^\wedge$. 
We can rigidify this situation by considering additional structure on the homology of $X$.
For any $\bb{E}_\infty$-ring $R$, the assignment $X \mapsto R[X]$ canonically refines to a functor 
$\cl{S} \to \rm{cCAlg}(\rm{Mod}_R)$ where the right hand side is the category of
$\bb{E}_\infty$-coalgebras in $\rm{Mod}_R$ and the coalgebra structure on $R[X]$ is induced
by the diagonal map $X \to X \times X$. Our main result is that, if $X$ is connected, we can recover
the coalgebra $\S[X]^\wedge_p$ from the coalgebra $\F_p[X]$ in a functorial way.
To state this in the correct generality, let us fix some notation.
Let $k$ be a perfect $\F_p$-algebra and denote by $\bb{W}(k)$ the spherical
Witt vectors of $k$ as in~\cite{ellII}. For an $\bb{E}_\infty$-ring $R$ denote by $\rm{Mod}_R^\wedge$
the category of $p$-complete $R$-modules equipped with the symmetric monoidal structure given by 
the $p$-completed tensor product. 

\newcounter{tmp}
\begingroup
\setcounter{tmp}{\value{theorem}}
\setcounter{theorem}{0} 
\renewcommand\thetheorem{\Alph{theorem}}

\begin{theorem}\label{result}
  Let $X\in \cl{S}$ be either connected or compact. The space of lifts of
  $k[X] \in \cCAlg(\Mod_k)$ to a coalgebra in $p$-complete, connective $\bb{W}(k)$-modules 
  is contractible and the unique point is given by $\bb{W}(k)[X]^\wedge_p$. 
  Moreover, for any $A \in \rm{cCAlg}(\Mod_{\bb{W}(k)}^{\wedge,\rm{cn}})$
  the base change along the map $\bb{W}(k)\to k$ induces an equivalence
  \[\rm{Map}_{\rm{cCAlg}(\rm{Mod}_{\bb{W}(k)}^{\wedge})}(A, \bb{W}(k)[X]^\wedge_p)
    \rar{\sim} \rm{Map}_{\rm{cCAlg}(\rm{Mod}_{k})}(A\otimes k, k[X]). \]
\end{theorem}

Since $\bb{W}(k)$ is a connective $\bb{E}_\infty$-ring, and thus a limit of square zero extensions of
$\pi_0\bb{W}(k)$, which is in turn a limit of square zero extensions of $\pi_{0}\bb{W}(k)/p\simeq k$,
we approach Theorem~\ref{result} by first understanding how to lift coalgebras along square zero
extensions of connective $\bb{E}_\infty$-rings. The deformation theory of $\bb{E}_\infty$-coalgebras
has thus far mainly been investigated by Lurie in~\cite{ellII} by embedding the category of 
\textit{flat} coalgebras fully faithfully into the category of sheaves on $\CAlg(\Sp)^{\cn}$. 
This provides no insight towards approaching Theorem~\ref{result},
as the $k$-homology of a space is never flat over $k$, unless the space is discrete. 
Our approach is instead the following: Suppose we have any map of connective $\bb{E}_\infty$-rings
$q: R\p \to R$. Colimits of coalgebras are computed underlying and the base change $q^\pt$
commutes with colimits and preserves connective modules. Moreover, the category $\cCAlg(\cl{C})$ 
is presentable whenever $\cl{C}$ is a presentably symmetric monoidal category. 
Thus, the adjoint functor theorem provides us with an adjunction

\begin{equation*}\label{adj}
\begin{tikzcd}
	{q^\ast:\rm{cCAlg}(\rm{Mod}_{R\p}^{\rm{cn}})} & {\mathrm{cCAlg}(\mathrm{Mod}_R^{\mathrm{cn}}):q_\ast}
	\arrow[""{name=0, anchor=center, inner sep=0}, shift left=2, from=1-1, to=1-2]
	\arrow[""{name=1, anchor=center, inner sep=0}, shift left=2, from=1-2, to=1-1]
	\arrow["\dashv"{anchor=center, rotate=-90}, draw=none, from=0, to=1].
\end{tikzcd}
\end{equation*}

We call $q_\pt$ the (connective) \textit{Weil restriction} along the map $q:R\p \to R$
\footnote{Such an adjunction exists without the connectivity assumptions, which we call
\textit{non-connective} Weil restriction. See Example~\ref{adjoint} for a more in-depth discussion.}.
Note that, since restriction of scalars is only \textit{lax} monoidal, an $R$-coalgebra $A$ 
has no underlying $R\p$-coalgebra, and so the "restriction" of $A$ along $q$ via the right 
adjoint $q_\pt$ is something of interest. These right adjoints are hard to describe in
general, but turn out to be essential for the deformation theory of coalgebras.
More precisely, in Definition~\ref{etale} we describe a full subcategory of
$\rm{cCAlg}(\Mod_R)^\rm{cn}$ on which the unit map $q^\pt q_\pt \to \id$ is an
equivalence for any square zero extension $R\p \to R$.
We call these coalgebras \textit{formally \'etale} and denote the subcategory 
with the superscript $(\blank)^{\rm{f\acute{e}t}}$. We show that Weil restriction 
of formally \'etale coalgebras along $\bb{W}(k)\to k$ behaves like a coalgebraic version 
of the spherical Witt vector construction.

\begin{theorem}\label{mainthm}
    Denote by $q:\bb{W}(k)\to k$ the natural augmentation. 
    Weil restriction along $q$ induces a fully faithful functor
    \[\cl{W}:\rm{cCAlg}(\Mod_{k}^{\rm{cn}})^{\rm{\fet}}
    \to \rm{cCAlg}(\Mod_{\bb{W}(k)}^{\cn, \wedge})^{\fet}.\] 
    Moreover, $\cl{W}(A)$ is, up to contractible choice, the unique coalgebra in connective
    $p$-complete $\bb{W}(k)$-modules with an equivalence $\cl{W}(A)\otimes k \simeq A$.
\end{theorem}

This is proven inductively, by first lifting along the $p$-completion tower from $k$ to 
$\pi_0\W(k)$ and then along the Postnikov tower to $\W(k)$.
Theorem~\ref{result} is immediate from Theorem~\ref{mainthm} once we know that
the $k$-homology of a space is formally \'etale. Let us make precise the notion of formally 
\'etale coalgebra before we proceed.

\endgroup

\setcounter{theorem}{\thetmp} 

\begin{construction}
Let $R$ be a connective $\E_\infty$-ring, $M$ a connective $R$-module and denote by 
$e:R\to R\oplus M$ the 0-section of the split square zero extension. For any 
$A\in \cCAlg(\Mod_R^{\cn})$ we define the \textit{universal $M$-deformation coalgebra} of $A$
as the Weil restriction
\[ \Omega^\infty_A M := e_\pt e^\pt A \in \cCAlg(\Mod_R^{\cn}).\]
From the adjunction we get a natural unit map $\eps: A \to \Omega^\infty_A M$. Moreover, the map
$R\oplus M \to R$ induces a section $\pi: \Omega^\infty_A M \to A$.
\end{construction}

We show that deformation problems involving $A$ are equivalent to giving lifts of the form
\[\begin{tikzcd}
      & {\Omega^\infty_A M} \\
    B & A
    \arrow[from=2-1, to=2-2]
    \arrow["{\pi}", from=1-2, to=2-2]
    \arrow[dashed, from=2-1, to=1-2]
\end{tikzcd}\]
  If we ask that these be uniquely solvable for all $M$ and $B$, the Yoneda lemma leads us to write
  down our main definition.
  
\begin{definition}\label{etale}
A coalgebra $A\in \rm{cCAlg}(\Mod_R^{\cn})$ is called \textit{formally \'etale} if the unit map
$\varepsilon:A \to \Omega^{\infty}_{A}M$ is an equivalence for all $M \in \Mod_R^{\cn}$.
We denote by $\rm{cCAlg}(\Mod_R^{\cn})^{\fet}$ the full subcategory 
spanned by the formally \'etale coalgebras.
\end{definition}

Note that for an algebra $S\in \CAlg(\Mod_R^{\cn})$, trying to impose the analogous
criterion would be asking that the inclusion $S \to S \oplus S\otimes_R M$ is an equivalence,
which is only true if $S=0$. Hence, this may not seem like a reasonable definition at first glance.
To reassure the reader that we are not only talking about the trivial coalgebra, we prove the
following general claim.

\begin{proposition}
    Let $A \in \cCAlg(\Mod_R^{\cn})$ be dualizable such that for the $R$-linear dual 
    $A^\vee\in \CAlg(\Mod_R)$ the relative cotangent complex $L_{A^\vee/R}$ vanishes. 
    Then $A$ is formally \'etale.
\end{proposition}

This allows us to prove Theorem~\ref{result} for compact spaces. To access arbitrary connected
spaces, the crucial step is to show that the homology of Eilenberg-MacLane spaces is formally \'etale.
To this end, we utilize recent work of Bachmann\textendash Burklund. In~\cite{bb} they show that for any field
$k$ of characteristic $p$ and any $V\in \Mod_{\F_p}^{\heartsuit}$, we have for $n\geq 0$ pullback
squares in $\cCAlg(\Mod_k)$ of the form

   \[\begin{tikzcd}
	{k[\Omega^\infty\Sigma^n V]} & {C_k(\Sigma^n k \otimes V)} \\
	k & {C_k(\Sigma^n k \otimes V)}
	\arrow[from=1-1, to=1-2]
	\arrow["{F-1}", from=1-2, to=2-2]
	\arrow[from=1-1, to=2-1]
	\arrow[', from=2-1, to=2-2].
\end{tikzcd}\] 

Here $C_k$ denote the cofree coalgebra and $F-1$ is the Artin \textendash Schreier map.
This is the dual picture to the pushout square on cohomology of Mandell in~\cite{mandell}, which
was utilized by Lurie in~\cite{dag8} to prove that the $\F_p$-cohomology of an arbitrary space 
is formally \'etale. The pullback diagram tells us that homology of Eilenberg \textendash MacLane spaces 
is cofree up to one "co-relation". Since Weil restriction preserves cofree coalgebras,
we show that the co-relation $F=1$ forces the coalgebra $k[\Omega^\infty \Sigma^n V]$ to
be formally \'etale, which allows us to prove our third theorem.

\begingroup
\setcounter{tmp}{\value{theorem}}
\setcounter{theorem}{2} 
\renewcommand\thetheorem{\Alph{theorem}}

\begin{theorem}\label{etresult}
Let $X \in \cl{S}$ be either connected or compact. Then for any $\F_p$-algebra $R$, the $R$-homology 
$R[X] \in \cCAlg(\Mod_R^{\cn})$ is formally \'etale.
\end{theorem}

\endgroup
\setcounter{theorem}{\thetmp} 
We strongly believe, but are unable to prove at this time, that the connectivity assumption
can be dropped for arbitrary spaces, not just compact ones. 
Recall that a space is called \textit{$p$-complete} if it is local with respect to the functor
$\F_p[\blank]:\cl{S}\to \Mod_{\F_p}$ and denote by $\cl{S}_p\subseteq \cl{S}$ the full 
subcategory spanned by the $p$-complete spaces. Moreover, call a space \textit{nilpotent}
if is connected and $\pi_1$ is a nilpotent group which acts nilpotently on the higher homotopy
groups. Write $\cl{S}_p^{\rm{nilp}}$ for the category of $p$-complete nilpotent spaces.
Combining Theorem~\ref{result} with~\cite[][Theorem 1.2]{bb}
we also deduce the following spherical lift of Mandells theorem.

\begin{corollary}\label{emb}
 Let $k$ be a perfect, separably closed field of characteristic $p$ with spherical
 Witt vectors $\W(k)$. The $p$-complete $\W(k)$-homology functor
 \[ \cl{S}_p^{\rm{nilp}} \to \rm{cCAlg}(\Mod^{\wedge}_{\bb{W}(k)})
 \quad X \mapsto \bb{W}(k)[X]^\wedge_p\]
 is fully faithful. In particular, for any nilpotent space $X$ we have a natural equivalence
 \[ X^\wedge_p\simeq \rm{Map}_{\rm{cCAlg}_{\bb{W}(k)}^\wedge}(\bb{W}(k), \bb{W}(k)[X]^\wedge_p).\]
\end{corollary}

It would be pleasant if we could replace $\W(k)$ by $\S_p^\wedge$ in Corollary~\ref{emb}.
On the nose this not possible as we need to keep track of descent data. In upcoming work, 
we plan to identify this descent datum with the \textit{coalgebra Frobenius}, a dual notion to the
Tate Frobenius of~\cite{tch} and construct and embedding of $\cl{S}^{\rm{nilp}}_p$ into coalgebras
$\cCAlg(\Sp^\wedge_p)$ equipped with a trivialization of the Frobenius action. This would dualize
work of Yuan in~\cite{yuan} and allow us to drop more finiteness assumptions on the spaces we input.

\subsection*{Outline}

We proceed along the following structure:
In Section 1 we recall the definition and basic properties of $\bb{E}_{\infty}$-coalgebras
and discuss some relevant coalgebraic right adjoints.
In Section 2 we review the setup of deformation theory developed by Lurie
in~\cite{dag14} and~\cite{ha}. We recall the notions of square zero extensions of
$\bb{E}_{\infty}$-rings and discuss cohesive and nilcomplete functors.
We then define the tangent complex of a cohesive functor and give proofs of some important facts
about its behavior.
In Section 3 we prove that the moduli of coalgebras are cohesive and thus we can analyze them using
the machinery discussed in the previous section. We then introduce formally \'etale coalgebras
and prove that they can be lifted uniquely and functorially against square zero extensions.
In Section 4 we define spherical and $p$-typical Witt vector style functors and prove that the
homology of any connected space is formally \'etale.

\subsection*{Conventions}
 Throughout the text, we use the following conventions:
\begin{itemize}
\item[(1)]  By category we always mean $(\infty,1)$-category and refer to $(1,1)$-categories as 
            $1$-categories. The text is \textit{model agnostic}, that is we make
            no reference to any specific model for the theory of $(\infty, 1)$-categories.
\item [(2)] We choose three nested Grothendieck universes $\cl{U} \subseteq \cl{V} \subseteq \cl{W}$,
            and refer to categories built from them as small categories, categories and
            large categories, respectively. We denote by $\rm{Cat}_{\infty}$ the large category
            of categories and disregard size issues from here on out.
\item[(3)] We denote the category of spaces by $\cl{S}$ and the category of spectra by $\rm{Sp}$.
\item[(4)] If $A, B$ are objects in some category $\cl{C}$, we use the words map and morphism
           $A\to B$ interchangeably to mean a point in the mapping space
           $\rm{Map}_{\cl{C}}(A,B) \in \cl{S}$. If moreover $\cl{C}$ is a stable category,
           we regard it as enriched over the category of spectra and write
           $\rm{map}_{\cl{C}}(A,B) \in \rm{Sp}$ for the mapping spectrum.
\item[(6)] By (co)algebra we always mean $\bb{E}_{\infty}$-(co)algebra. If we want to refer to
           a 1-categorical version of some gadget we call them \textit{discrete}.
\end{itemize}

\subsection*{Acknowledgments}
 This paper is based on my Master's thesis which was advised by Thomas Nikolaus and Achim Krause. I would
 like to thank them for guiding this project and being wonderful teachers throughout the years. I want
 to especially thank Achim Krause for many insightful discussions and his tremendous
 patience with my questions. The proof of the final results added in this version emerged from
 discussions with Robert Burklund and Tom Bachmann and utilizes their recent work on coalgebras.
 I am grateful for their input and fearless treatment of pro-objects. I was supported by the Danish
 National Research Foundation through the Copenhagen Center for Geometry and Topology (DNRF 151) for
 the later revisions on this work.
\addtocontents{toc}{\protect\setcounter{tocdepth}{2}}
\section{Coalgebras}
In this section, we recall the definition of an $\bb{E}_{\infty}$-coalgebra
in a symmetric monoidal category $\cl{C}$ and collect some facts about it.
Colimits of coalgebras are computed underlying and if $\cl{C}$ is presentable and the tensor
product commutes with colimits in each variable separately, then $\rm{cCAlg}(\cl{C})$ is again
presentable. For a map of connective $\bb{E}_\infty$-rings $f:R \to S$ this lets us introduce the adjunction
\[\begin{tikzcd}
	{\rm{cCAlg}(\rm{Mod}_R^{\rm{cn}})} & {\rm{cCAlg}(\rm{Mod}_S^{\rm{cn}})}
	\arrow[""{name=0, anchor=center, inner sep=0}, "{{f_\pt}}", shift left=2, from=1-2, to=1-1]
	\arrow[""{name=1, anchor=center, inner sep=0}, "{{f^\ast}}", shift left=2, from=1-1, to=1-2]
	\arrow["\dashv"{anchor=center, rotate=-90}, draw=none, from=1, to=0]
\end{tikzcd}\]
where $f^\pt$ is the base change and $f_\pt$ is conjured up by the adjoint functor theorem. Crucially,
$f_\pt$ is \textit{not} induced from the restriction of scalars on the level of module categories. 
We then consider the coalgebra structure on the $R$-chains of a space $X$ and the algebra structure on
the dual of a coalgebra in a presentably monoidal category. 

\subsection{Generalities}

\begin{definition}
  Let $\cl{C}$ be a symmetric monoidal category. We denote by $\rm{CAlg}(\cl{C})$ the category
  of $\bb{E}_{\infty}$-algebras in $\cl{C}$. For $\cl{C}= \rm{Sp}$ we write $\rm{CAlg}(\rm{Sp})= \rm{CAlg}$
  and refer to objects of $\rm{CAlg}$ simply as $\bb{E}_{\infty}$-rings.
\end{definition}

\begin{proposition}[Lurie]\label{calg}
  Let $(\cl{C}, \otimes)$ be a symmetric monoidal category, then the following statements hold:
  \begin{enumerate}
    \item The forgetful functor $U:\rm{CAlg}(\cl{C}) \to \cl{C}$ is conservative and
          commutes with limits.
    \item The coproduct of two algebras $R,S \in \rm{CAlg}(\cl{C})$ is given by the tensor
          product $R \otimes S$.
    \item If $\cl{C}$ is presentable and $\blank \otimes \blank$ commutes with colimits in both variables
          separately, then $\rm{CAlg}(\cl{C})$ is presentable as well.
  \end{enumerate}
\end{proposition}

\begin{proof}
  The first claim is a combination of~\cite[][Lemma 3.2.2.6]{ha} and~\cite[][Corollary 3.2.2.5]{ha}.
  The second is shown in~\cite[][Corollary 3.2.4.7]{ha} and the third is~\cite[][Corollary 3.2.3.5.]{ha}.
\end{proof}

\begin{definition}
  Let $\cl{C}$ be a symmetric monoidal category. The opposite category $\cl{C}\op$ inherits 
  a natural symmetric monoidal structure and we define the category of $\bb{E}_{\infty}$-coalgebras
  in $\cl{C}$ as
  \[\rm{cCAlg}(\cl{C}):= \rm{CAlg}(\cl{C}^{\rm{op}})^{\rm{op}}.\]
\end{definition}

\begin{remark}\label{colimits}
  In particular, a coalgebra $A \in \rm{cCAlg}(\cl{C})$ comes equipped with the datum of a 
  ``coherently commutative'' comultiplication and counit maps
  \begin{align*}
    \Delta_{A}: A &\to (A \otimes A)^{h\Sigma_2} \\
    \eta: A &\to 1_{\cl{C}},
  \end{align*}
  where $1_{\cl{C}}$ denotes the unit of $\cl{C}$. Note that, in general, there is no way
  to describe $\rm{cCAlg}(\cl{C})$ as a category of algebras in some suitable category $\cl{D}$.
  Nonetheless, Proposition~\ref{calg} immediately implies the following:
  \begin{enumerate}
    \item The forgetful functor $U:\rm{cCAlg}(\cl{C}) \to \cl{C}$ is conservative and commutes with colimits.
    \item The product of two coalgebras $R,S \in \rm{cCAlg}(\cl{C})$ is given by $R \otimes S$.
  \end{enumerate}
However, we cannot deduce presentability this way since the opposite of a presentable category is 
almost never presentable. 
\end{remark}

\begin{proposition}[Lurie]\label{present}
  Let $\cl{C}$ be a symmetric monoidal category such that $\blank \otimes \blank$ commutes with
  colimits in each variable separately and $\cl{C}$ is presentable. Then, the category
  $\rm{cCAlg}({\cl{C}})$ is presentable.
\end{proposition}
\begin{proof}
  This is~\cite[][Corollary 3.1.4.]{ellI}.
\end{proof}

This can be seen as an analogue of the classical theorem by Sweedler that every coalgebra
in the 1-category of vector spaces over a field is a filtered colimit of its finite dimensional
sub-coalgebras, see~\cite{sweedler1969hopf}. However, unlike in the classical situation, if $\cl{C}$
is $\kappa$-presentable, we only deduce that $\rm{cCAlg}(\cl{C})$ is $\tau$-presentable for some 
$\tau \geq \kappa$. This is one of the main defects the category of coalgebras has over that
of algebras. 

\begin{definition}
    Let $\cl{C}$ be as in Proposition~\ref{present}. Since by Remark~\ref{colimits} the
    forgetful functor $U:\cCAlg(\cl{C})\to \cl{C}$ commutes with colimits, it admits
    a right adjoint 
    \[C:\cl{C}\to \cCAlg(\cl{C}).\]
    For any $M\in \cl{C}$ we call $C(M)$ the \textit{cofree coalgebra} on $M$.
\end{definition}

\begin{lemma}\label{limits}
  Let $\rm{Cat}_{\infty}^{\otimes}$ denote the (large) category of symmetric monoidal 
  categories and monoidal functors. Then the functors
  \[\rm{Cat}_{\infty}^{\otimes} \to \rm{Cat}_{\infty} \quad \cl{C}\mapsto \rm{CAlg}(\cl{C}),\]
  \[\rm{Cat}_{\infty}^{\otimes} \to \rm{Cat}_{\infty} \quad \cl{C} \mapsto \rm{cCAlg}(\cl{C})\]
  commute with limits.
\end{lemma}

\begin{proof}
  By Proposition~\ref{calg} the functor $\rm{CAlg}(\blank)$ factors through the (large) category 
  $\rm{Cat}_{\infty}^{\amalg}$ of categories which admit finite coproducts and functors which
  preserve finite coproducts. As such it admits a left adjoint which equips $\cl{D} \in \rm{Cat}_{\infty}^{\amalg}$ with the
  cocartesian monoidal structure. Moreover, the inclusion $\rm{Cat}_{\infty}^{\amalg} \rari{} \rm{Cat}_{\infty}$
  admits a left adjoint which takes a category $\cl{C}$ to the free finite coproduct completion, namely
  the full subcategory of $\rm{P}(\cl{C})$ spanned by finite coproducts of representables.
  Thus, both functors commute with limits, and so the composition $\CAlg(\blank)$ does as well. \\
  For the second functor, we simply observe that it is given by the composition 
  \[ \rm{Cat}_{\infty}^{\otimes} \rar{(\blank)\op} \rm{Cat}_{\infty}^{\otimes} \rar{\rm{CAlg(\blank)}} \rm{Cat}_{\infty}
  \rar{(\blank)\op}\rm{Cat}_{\infty},\]
which immediately implies the claim, since taking the opposite category is an involution, i.e.~an
equivalence of categories and thus commutes with limits.
\end{proof}

\subsection{Coalgebraic right adjoints and duality}

\begin{definition}[Lurie tensor product]
  Let $\rm{Pr^{L}}$ denote the (large) category of presentable categories
  with maps given by functors which commute with colimits. 
  By~\cite[][Proposition 4.8.1.15.]{ha} $\rm{Pr^{L}}$ inherits a natural
  symmetric monoidal structure, such that we have a map $\cl{C}\times \cl{D}\to \cl{C}\otimes \cl{D}$
  exhibiting $\rm{Fun}^{\rm{L}}(\cl{C}\otimes \cl{D},\cl{E})$ as the category of functors
  $\cl{C}\times \cl{D}\to \cl{E}$ which commute with colimits in each variable separately.
  We call $\rm{CAlg}(\rm{Pr^{L}})$ the category of \textit{presentably symmetric monoidal}
  categories.
\end{definition}

\begin{example}
  If $R$ is an $\bb{E}_{\infty}$-ring then the category of $R$-modules in spectra $\rm{Mod}_{R}$
  and of connective $R$-modules $\Mod_R^{\cn}$ are both presentable. 
\end{example}

\begin{definition}
By Proposition~\ref{present} the assignment $ \cl{C} \mapsto \cCAlg(\cl{C})$ defines a functor
\[ \rm{cCAlg}(\blank): \CAlg(\rm{Pr^L})\to \rm{Pr^L}.\]
 In particular, for any map $f:\cl{C} \to \cl{D}$  in $\CAlg(\rm{Pr^L})$ we get an adjunction 
\[\begin{tikzcd}
	{\rm{cCAlg}(\cl{C})} & {\rm{cCAlg}(\cl{D})}
	\arrow["{f^\pt}", shift left, from=1-1, to=1-2]
	\arrow["{f_\pt}", shift left, from=1-2, to=1-1].
\end{tikzcd}\]
Let $g:\cl{D}\to \cl{C}$ denote the right adjoint of $f$. The functor $f^\pt$ is induced by $f$, 
that is for any $A\in \cCAlg(\cl{C})$ the underlying object of $f^\pt(A)$
is given by $f(A)$. However, it is in general \textit{not true} that on underlying objects $f_\pt$ agrees with $g$.
Indeed, $g$ need only be \textit{lax symmetric monoidal} but $f_\pt$ preserves products of coalgebras which are given
by the monoidal products of $\cl{C}$ and $\cl{D}$. We call $f_\pt$ the \textit{coalgebraic right adjoint} of $f$.
\end{definition}

\begin{construction}\label{radj}
    Let $\cl{C},\cl{D} \in \rm{CAlg}(\rm{Pr^L})$ be stable and equipped with a compatible $t$-structure.
    Then $\cl{C}_{\geq 0}$ is closed under the tensor product and thus the inclusion
    \[ \cl{C}_{\geq 0} \rari{} \cl{C}\]
    is a map in $\CAlg(\rm{Pr^L})$. Thus it admits a coalgebraic right adjoint 
    \[ \tau_{\geq 0}^{c}: \rm{cCAlg}(\cl{C}) \to \cCAlg(\cl{C}_{\geq 0}).\]
    Since connective coalgebras form a full subcategory, it is true that $\tau_{\geq 0}(A) \simeq A$ whenever
    $A$ is connective but in general, $\tau_{\geq 0}^c(A)$ does not agree with the underlying connective cover of $A$.
    Moreover, if we have a map $f:\cl{C}\to \cl{D}$ in $\CAlg(\rm{Pr}^L)$ which preserves connective objects, then
    we get a commuting diagram of coalgebraic right adjoints
   \[\begin{tikzcd}
	{\rm{cCAlg}(\cl{C})} & {\rm{cCAlg}(\cl{D})} \\
	{\rm{cCAlg}(\cl{C}_{\geq 0})} & {\rm{cCAlg}(\cl{D}_{\geq 0})}
	\arrow["{f_!}"', from=1-2, to=1-1]
	\arrow["{\tau_{\geq 0}^c}", from=1-2, to=2-2]
	\arrow["{\tau^c_{\geq 0}}"', from=1-1, to=2-1]
	\arrow["{f_\ast}", from=2-2, to=2-1].
\end{tikzcd}\] 
Note that, for $A \in \cCAlg(\cl{D}_{\geq 0})$, whenever $f_!(A)$ happens to be connective,
it necessarily agrees with $f_\pt (A)$. 
\end{construction}

\begin{example}[Weil restriction]\label{adjoint}
    Let $f:R \to S$ be a map of connective $\E_\infty$-rings. The base
    change functor $f^\pt: \rm{Mod}_R \to \rm{Mod}_S$ is symmetric monoidal,
    commutes with colimits and preserves connective objects. Thus the restriction to connective modules
    induces a coalgebraic adjunctions
   \[\begin{tikzcd}
	{f^\pt:\rm{cCAlg}(\Mod_R)} & {\rm{cCAlg}(\Mod_S):f_!}
	\arrow[""{name=0, anchor=center, inner sep=0}, shift left=2, from=1-1, to=1-2]
	\arrow[""{name=1, anchor=center, inner sep=0}, shift left=2, from=1-2, to=1-1]
	\arrow["\dashv"{anchor=center, rotate=-90}, draw=none, from=0, to=1]
\end{tikzcd}\] 
   \[\begin{tikzcd}
	{f^\pt:\rm{cCAlg}(\Mod_R^{\rm{cn}})} & {\rm{cCAlg}(\Mod_S^{\rm{cn}}):f_\pt}
	\arrow[""{name=0, anchor=center, inner sep=0}, shift left=2, from=1-1, to=1-2]
	\arrow[""{name=1, anchor=center, inner sep=0}, shift left=2, from=1-2, to=1-1]
	\arrow["\dashv"{anchor=center, rotate=-90}, draw=none, from=0, to=1].
\end{tikzcd}\] 
For $A\in \cCAlg(\Mod_R)$ we refer to $f_! A$ as the \textit{Weil restriction} and to $f_\pt \tau_{\geq 0}^c A$ 
as the \textit{connective Weil restriction} of $A$ along $f$. Observe that we have a commutative diagram
in $\rm{Pr^L}$
\[\begin{tikzcd}
	{\cCAlg(\Mod_R)} & {\cCAlg(\Mod_S)} \\
	{\Mod_R} & {\Mod_S}
	\arrow["{f^\pt}", from=1-1, to=1-2]
	\arrow[from=1-1, to=2-1]
	\arrow[from=1-2, to=2-2]
	\arrow["{f^\pt}"', from=2-1, to=2-2].
\end{tikzcd}\]
Thus, taking right adjoints gives a commutative diagram
\[\begin{tikzcd}
	{\cCAlg(\Mod_R)} & {\cCAlg(\Mod_S)} \\
	{\Mod_R} & {\Mod_S}
	\arrow["{f_!}"', from=1-2, to=1-1]
	\arrow["{C_R}", from=2-1, to=1-1]
	\arrow["{C_S}"', from=2-2, to=1-2]
	\arrow["{f_\pt}", from=2-2, to=2-1],
\end{tikzcd}\]
which tells us that Weil restriction commutes with the cofree coalgebra functors in the sense that
\[ f_!C_S(M) \simeq C_R(f_\pt M)\]
for any $M\in \Mod_S$ where $f_\pt M$ denotes the usual restriction of scalars. Note that we can draw the same
diagram where we replace all module categories by connective modules. Even though we care mainly about the
connective Weil restriction functor, the coalgebraic connective cover of the cofree coalgebra is even less
approachable than the cofree coalgebra itself, so we shall not utilize this. 
Instead, in Section~\ref{homocoalg} we will compute with the non-connective Weil-restriction 
and see that the result is connective.
\end{example}

\begin{example}[Duality]\label{predual}
  For any $\cl{C}\in \CAlg(\rm{Pr^L})$ and $X\in \cl{C}$,
  the functor $X\otimes \blank :\cl{C} \to \cl{C}$ commutes with colimits and thus admits a right adjoint
  $\rm{map}_{\cl{C}}(X, \blank)$.
  This assignment defines a functor $\cl{C}\op \to \rm{Fun}(\cl{C}, \cl{C})$, and we denote
  its adjoint as
  \[ \rm{map}_{\cl{C}}(\blank, \blank): \cl{C}^{\op} \times \cl{C} \to \cl{C}.\]
  The target of $(\blank)^\vee:=\rm{map}_\cl{C}(\blank, \mathds{1}_\cl{C}): \cl{C} \to \cl{C}\op$ 
  is not presentable, so this is not a map in $\CAlg(\rm{Pr^L})$. However, it still admits a coalgebraic 
  right adjoint via the following construction. The functor $\rm{map}_\cl{C}(\blank, \blank) $ is lax monoidal with 
  respect to the monoidal structure on $\cl{C}\op \times \cl{C}$, defined by
\[ (A,B)\otimes_{\cl{C}\op \times \cl{C}} (C,D):= (A \otimes C, B \otimes D).\]
Thus, it refines to a functor
\[ \rm{map}_{\cl{C}}(\blank, \blank): \rm{CAlg}(\cl{C}\op\times \cl{C})
  \simeq \rm{cCAlg}(\cl{C})\op \times \rm{CAlg}(\cl{C})
  \to \rm{CAlg}(\cl{C}).\]
In particular, for each $R \in \rm{CAlg}(\cl{C})$ we get a functor
\begin{align*}
  \rm{map}_{\cl{C}}(\blank, R):\rm{cCAlg}(\cl{C}) = \rm{CAlg}(\cl{C}\op)\op \to \rm{CAlg}(\cl{C})\op.
\end{align*}
Since colimits of coalgebras and limits of algebras are computed underlying, the assignment
\[A \mapsto A^\vee:=\rm{map}_{\cl{C}}(A, \mathds{1}_{\cl{C}}) \]
takes colimits of coalgebras to limits of algebras. By the adjoint functor theorem, we get an adjunction.
\[\begin{tikzcd}
	{(\blank)^\vee:\cCAlg(\cl{C})} & {\rm{CAlg}(\cl{C}):(\blank)^\circ}
	\arrow[""{name=0, anchor=center, inner sep=0}, shift left=2, from=1-1, to=1-2]
	\arrow[""{name=1, anchor=center, inner sep=0}, shift left=2, from=1-2, to=1-1]
	\arrow["\dashv"{anchor=center, rotate=-90}, draw=none, from=0, to=1].
\end{tikzcd}\]
For $A \in \rm{cCAlg}(\cl{C})$ and $R \in \rm{CAlg}(\cl{C})$. We call $A^\vee$ the \textit{dual algebra}
of $A$ and $R^\circ$ the \textit{pre-dual coalgebra} of $R$.
\end{example}

If the underlying object of $R\in \rm{CAlg}(\cl{C})$ is dualizable, then the dual $R^\vee$ also inherits a natural
coalgebra structure. In fact this construction agrees with $R^\circ$ as we will see in Corollary~\ref{dualad}.

\begin{proposition}\label{duality}
Let $\cl{C}$ be a symmetric monoidal category and denote by $\cl{C}^{\rm{dual}}$ the full subcategory
spanned by the dualizable objects. The functor
  \[ \rm{cCAlg}(\cl{C}^{\rm{dual}})\op \rar{\sim} \rm{CAlg}(\cl{C}^{\rm{dual}})
  \quad A \mapsto A^{\vee}\]
  is an equivalence of categories with inverse taking $R\in \rm{CAlg}(\cl{C}^{\rm{dual}})$
  to the dual $R^{\vee}$ with the induced coalgebra structure.
\end{proposition}
\begin{proof}
This is immediate from \cite[][Proposition 3.2.4]{ellI}.
\end{proof}

For the next lemma we employ the following terminology:\\
For a functor $F:\cl{C}\to \cl{D}$ we say that $A \in \cl{C}$ is $F$-local, if the natural
transformation $\rm{Map}_{\cl{C}}(A, \blank) \to \rm{Map}_{\cl{D}}(F(A), F(\blank))$
is an equivalence.
\begin{lemma}\label{laxlocal}
  Let $\cl{C}, \cl{D}$ be symmetric monoidal $\infty$-categories and $F:\cl{C}\to \cl{D}$ be a lax symmetric
  monoidal functor. Suppose we have $R\in \rm{CAlg}(\cl{C})$ such that each tensor power
   $R^{\otimes n}$ considered as an object in $\cl{C}$ is $F$-local. Then the map
  \[ \rm{Map}_{\rm{CAlg(\cl{C})}}(R, S)\rar{F} \rm{Map}_{\rm{CAlg}(\cl{D})}(F(R), F(S)) \]
  is an equivalence.
\end{lemma}
\begin{proof}
  Since $F$ is lax symmetric monoidal, it induces a map of $\infty$-operads
  \[\begin{tikzcd}
	{\cl{C}^\otimes} && {\cl{D}^\otimes} \\
	& {\rm{Fin}_\pt}
	\arrow["p"', from=1-1, to=2-2]
	\arrow["q", from=1-3, to=2-2]
	\arrow["f", from=1-1, to=1-3]
\end{tikzcd}\]
which takes any $(A_{1}, \dots, A_{n})\in \cl{C}^{\otimes}_{\langle n \rangle}$ to the point
$(F(A_{1}), \dots, F(A_{n}))\in \cl{D}^{\otimes}_{\langle n \rangle}$. A commutative algebra structure on an object
$R\in \cl{C}$ is precisely given by a a section $s_{R}$ of $p$ which takes $\langle n \rangle $ to
$(R, \dots, R) \in \cl{C}_{\langle n \rangle }$ and maps inert morphisms to inert morphisms. Let
$\varphi: \langle n \rangle \to \langle m \rangle $ be a morphism in $\rm{Fin}_{\pt}$ and denote by $p_{i}: \langle n \rangle \to \langle 1 \rangle$ the unique inert
map which sends $i \mapsto 1$. For each $i$ we have a factorization
\[ \langle n \rangle \rar{\psi_{i}} \langle k_{i}\rangle \rar{\pi_{i}} \langle 1 \rangle\]
of $p_{i}\circ \varphi$ into an inert map $\psi_{i}$ and an active map $\pi_{i}$.
Then for each $B= (B_{1}, \dots B_{m})\in \cl{C}^{\otimes}$ we get equivalences
\begin{align*}
  \rm{Map}^{\varphi}_{\cl{C}^{\otimes}}(s_{R}(\langle n \rangle), B)
  &\simeq \prod_{i =1, \dots, m}\rm{Map}^{ p_{i}\circ \varphi}_{\cl{C}^{\otimes}}((R, \dots, R), B_{i})\\
  &\simeq \prod_{i =1, \dots, m} \rm{Map}_{\cl{C}}(R^{\otimes k_{i}}, B_{i})\\
  &\simeq \prod_{i =1 ,\dots, m} \rm{Map}_{\cl{D}}(F(R)^{\otimes k_{i}}, F(B_{i}))\\
  &\simeq \prod_{i =1 , \dots m } \rm{Map}^{p_{i} \circ \varphi}_{\cl{D}^{\otimes}}((f \circ s_{R})(\langle n \rangle), B_{i})\\
  &\simeq \rm{Map}_{\cl{D}^{\otimes}}^{\varphi}((f \circ s_{R})(\langle n \rangle), B),
\end{align*}
hence each value of $s_{R}: \rm{Fin}_{\pt}\to \cl{C}^{\otimes}$ is $f$-local, and thus $s_{R}$
is local with respect to the functor $f_{\pt}: \Gamma(p) \to \Gamma(q)$. Since $\rm{CAlg}(\cl{C})$
and $\rm{CAlg}(\cl{D})$ are full subcategories of $\Gamma(p)$ and $\Gamma(q)$ respectively, the
claim follows.
\end{proof}

\begin{corollary}\label{dualad}
Let $\cl{C}$ be a symmetric monoidal category and $A,B \in \rm{cCAlg}(\cl{C})$ with $A$ dualizable.
The natural map
\[ \rm{Map}_{\rm{cCAlg}(\cl{C})}(B,A)\to\rm{Map}_{\rm{CAlg}(\cl{C})}(A^{\vee}, B^{\vee})\]
is an equivalence. In particular, we have that $(A^\vee)^\circ \simeq A$.
\end{corollary}
\begin{proof}
 This is immediate by applying Lemma~\ref{laxlocal} to the duality functor 
 $(\blank)^{\vee}:\cl{C}\op \to \cl{C}$ since dualizable objects are closed under tensor products.
\end{proof}

We now review why the homology of a space carries an $\bb{E}_{\infty}$-coalgebra structure. The following
lemma in particular shows that every $X\in \cl{S}$ admits a unique coalgebra structure with respect 
to the cartesian product induced from the diagonal map $X\rar{\Delta} X\times X$.

\begin{lemma}\label{trivcalg}
  Let $\cl{C}$ be a category which admits finite products equipped with the cartesian monoidal structure.
  Then the forgetful functor $\rm{cCAlg}(\cl{C}) \to \cl{C}$ is an equivalence.
\end{lemma}
\begin{proof}
  By~\cite[][Corollary 2.4.4.10.]{ha} the map $\rm{CAlg}(\cl{C}\op) \to \cl{C}\op $ is an equivalence, hence
  the claim follows by applying $(\blank)\op$.
\end{proof}

The following generalization of the Eilenberg \textendash Zilber Theorem equips every ``generalized suspension''
of a space with the structure of an $\bb{E}_{\infty}$-coalgebra.

\begin{proposition}\label{chains}
  Let $\cl{C}$ a presentably symmetric monoidal category. The functor
  \begin{align*}
    \cl{S} \to \cl{C} \qquad X \mapsto 1_{\cl{C}}[X]
  \end{align*}
  which sends a space $X$ to the colimit over the constant diagram $X \to \pt \rar{1_{\cl{C}}} \cl{C}$
  is symmetric monoidal with respect to the cartesian monoidal structure on $\cl{S}$.
  In particular, this induces a functor
  \[ \cl{S} \simeq \rm{cCAlg}(\cl{S}) \to \rm{cCAlg}(\cl{C}) \quad X \mapsto 1_{\cl{C}}[X]\]
\end{proposition}

\begin{proof}
  Since $\blank \otimes \blank$ commutes with colimits in both variables
  separably by assumption, we have that:
  \begin{align*}
    (\colim_X 1_{\cl{C}})\otimes (\colim_Y 1_{\cl{C}})
    \simeq \colim_X \colim_Y (\underbrace{1_{\cl{C}}\otimes 1_{\cl{C}}}_{\simeq 1_{\cl{C}}}) \simeq \colim_{X \times Y} 1_{\cl{C}}.
  \end{align*}
  The second statement then follows from Lemma~\ref{trivcalg}.
\end{proof}

\begin{example}\label{homology}
By Proposition~\ref{chains}, for each $\bb{E}_{\infty}$-ring $R$, the $R$-homology functor
\begin{align*}
  \cl{S} \to \rm{Mod}_{R} \quad X \mapsto R[X]
\end{align*}
is symmetric monoidal and thus refines to a functor
\[ R[\blank]: \cl{S}\simeq \rm{cCAlg}(\cl{S})\to \rm{cCAlg}(\Mod_{R}). \]
Hence, the $R$-homology of a space $X$ carries a natural coalgebra structure. Moreover,the $R$-cohomology of $X$
  \[ R[X]^\vee = \rm{map}_{\Mod_R}(R[X], R) \simeq \lim_{X}R = R^{X}\]
  inherits a natural $R$-algebra structure. If $X$ is finite, then $R[X]$ is dualizable and so by 
  Corollary~\ref{dualad} we have $(R^X)^\circ \simeq R[X]$.
\end{example}

\section{Review of deformation theory}
In this chapter we review the theory of square zero extensions and the deformation theory of functors
$X:\rm{CAlg}^{\rm{cn}} \to \cl{S}$, as developed in~\cite{ha} and~\cite{dag14}.
Let $R^\eta \to R$ be a square zero extensions of connective $\E_\infty$-rings with fiber
$M \in \rm{Mod}_{R}^{\rm{cn}}$. The goal is to understand the fiber
$\rm{fib}_{A}(X(R^\eta)\to X(R))$ over some point $A\in X(R)$.
The key insight is that $R^\eta\to R$ is classified by a map into
the split square zero extension $R \oplus M[1]\in \rm{cCAlg}_{/R}$. 
For well behaved functors $X:\rm{CAlg}^{\rm{cn}} \to \cl{S}$
the space $X(R^\eta)$ can then be obtained by pulling back along the induced map
$ X(R) \to X(R \oplus M[1])$. In this case the fiber $\rm{fib}_{A}(X(\widetilde{R})\to X(R))$ is given a path space in $X(R \oplus M[1])$. Moreover, for each $A\in X(R)$ writing
\[ X_{A}^{R\oplus M[n]}:=\rm{fib}_{A}(X(R\oplus M[n])\to X(R))\]
the sequence $\{X_{A}^{R\oplus M[n]}\}_{n\in \bb{N}}$ defines a spectrum
$T^{M}_{X_{A}}$ which collects the deformation theoretic data into one spectrum 
called the \textit{tangent complex} of $X$ at the point $A$. 

\subsection{Square zero extensions and deformations}
\begin{proposition}\label{loops}
  For every $\bb{E}_{\infty}$-ring $R$ there is an equivalence of categories
  \[\rm{Mod}_{R} \rar{\sim} \rm{Sp}(\rm{CAlg}_{/R})\]
  such that for each $n \geq 0$ the functor
  \[ \rm{Mod}_{R} \rar{\sim} \rm{Sp}(\rm{CAlg}_{/R}) \rar{\Omega^{\infty -n}} \rm{CAlg}_{/R} \]
  sends a module $M$ to an augmented $R$-algebra whose underlying $R$-module is given by
  the direct sum $R \oplus M[n]$.
\end{proposition}
\begin{proof}
This is~\cite[][Theorem 7.3.4.13]{ha}.
\end{proof}

\begin{remark}
  For a connective $\bb{E}_{\infty}$-ring $R$ and a connective $R$-module $M$ we call
  $ \Omega^{\infty} M =R \oplus M$ with the $R$-algebra structure described above the
  \textit{split square zero extension} of $R$ along $M$. If $R$ and $M$ are discrete, the
  multiplication is explicitly given by
  \[(a+m)(b +n) := ab + an +mb ,\]
  i.e.~ we have that $R \oplus R \simeq R[x]/x^{2}$.
\end{remark}

\begin{definition}
  For a connective $\bb{E}_{\infty}$-ring $R$ and a connective $R$-module $M$ we say that a map $R^{\eta} \to R$ is
  a \textit{square zero extension} of $R$ along $M$ if it fits into
  a pullback diagram
  \[\begin{tikzcd}
	{R^\eta} & R \\
	R & {R \oplus M[1]}
	\arrow[from=1-1, to=1-2]
	\arrow["{(\id, 0)}", from=1-2, to=2-2]
	\arrow[from=1-1, to=2-1]
	\arrow["{(\id,\eta)}"', from=2-1, to=2-2].
\end{tikzcd}\]
Moreover, we call the mapping space $\rm{Map}_{\rm{CAlg}_{/R}}(R, R \oplus M)$ the
space of \textit{derivations} $\eta:R \to M$.
\end{definition}
This definition make sense without any connectivity assumptions. However, since our interest
lies primarily in the case where everything is connective, we have chosen to include them
in the definition to avoid awkward terminology.

\begin{remark}\label{cohsq0}
  Note that, the split square zero extension is precisely the one classified by the zero derivation
  \[R \rar{(\id,0)} R \oplus M[1].\]
   Moreover, if $p:R^{\eta} \to R$ is any square zero extension classified by a derivation 
   $\eta:R \to M[1]$, then by taking fibers we get a commutative diagram with exact columns
  \[\begin{tikzcd}
	{M} & {M} \\
	{R^\eta} & R \\
	R & {R \oplus M[1]}
	\arrow[from=1-1, to=2-1]
	\arrow[from=2-1, to=3-1]
	\arrow["{(\id,\eta)}", from=3-1, to=3-2]
	\arrow["{(\id,0)}", from=2-2, to=3-2]
	\arrow["0",from=1-2, to=2-2]
	\arrow["\sim", from=1-1, to=1-2]
	\arrow[from=2-1, to=2-2].
\end{tikzcd}\]
In particular, the fiber of $f$ inherits a natural $R$-module structure, such that the multiplication
map factors as
\[ M \otimes_{R^\eta} M \to M\otimes_R M \to M\]
where the second map is induced from the 0-map $M \to R$ by applying $\blank_{R}\otimes M$.
Since any map of the form $f_1 \otimes \cdots \otimes f_{n-1} \otimes 0$ admits a $\Sigma_n$-equivariant
nullhomotopy, we deduce that all coherent multiplication maps
\[ (M^{\otimes n})_{h\Sigma_n} \to M\]
are nullhomotopic. In this sense, square zero extensions of $\E_\infty$-rings are "coherently" square
zero and admit no nontrivial power operations on the fiber.
\end{remark}

\begin{remark}\label{torsor}
  For any connective $R$-module $M$ the augmented $R$-algebra $\Omega^{\infty}M = R \oplus M$ inherits the
  structure of an $\bb{E}_{\infty}$-monoid in $\rm{CAlg}_{/R}$ with delooping given by $R \oplus M[1]$.
  Thus, we can think of $R \oplus M[1]$ as the classifying object for square zero extensions with
  fiber $M$ the ``universal'' derivation being given by the trivial section $R \to R \oplus M[1]$.
  From this perspective, a square zero extension is precisely a $R\oplus M$-torsor in $\rm{CAlg}_{/R}$.
\end{remark}

To make this definition work for us we require a way to check in practice whether a given map
of $\bb{E}_{\infty}$-rings is a square zero extension. This is provided by the following proposition, which
in particular implies that for discrete rings our notion agrees with the classical definition of a
square zero extension.

\begin{proposition}
  Let $R\p\to R$ be a map of connective $\bb{E}_{\infty}$-rings with fiber $M$ such that
  $M \in \rm{Sp}_{\geq n} \cap \rm{Sp}_{\leq 2n}$ and the multiplication map $M \otimes_{R\p}M \to M$ is nullhomotopic.
  Then $R\p\to R$ is a square zero extension.
\end{proposition}
\begin{proof}
  This is immediate from~\cite[][Theorem 7.4.1.23.]{ha}.
\end{proof}

\begin{example}
  If $R\p \to R$ is a surjective map of ordinary commutative rings with kernel $M\subseteq R\p$,
  then $R\p \to R$ is a square zero extension if and only if $M^{2}=0$.
  In particular, for every $n$ the map $\Z/p^{n}\to \Z/p^{n-1}$ is square zero with kernel $\F_{p}$.
\end{example}

\begin{example}\label{sq0ex}
  If $R$ is any connective $\bb{E}_{\infty}$-ring, then the map $\tau_{\leq n}R \to \tau_{\leq n-1}R $ is a square zero
  extension with fiber $\pi_{n}R[n]$.
\end{example}

For a functor $X:\rm{CAlg}^{\rm{cn}} \to \cl{S}$ and a square zero extension $R^{\eta} \to R$
we want to study the fibers of the map $X(R^{\eta}) \to X(R)$, i.e.~given $A\in X(R)$ we wish
to understand the space of deformations of $A$ to an object $\widetilde{A}\in X(R^{\eta})$.
Notice that, since $R$ and $M$ are connective, the derivation
\[R \rar{(\rm{id}, \eta)} R \oplus M[1]\]
is surjective on $\pi_{0}$, indeed an isomorphism. This motivates the following definition for
a class of functors which are ``well behaved'' with regard to square zero extensions.

\begin{definition}
Let $\cl{C}$ be a category. A functor $X:\rm{CAlg}^{\rm{cn}}\to \cl{C}$ is called:
\begin{enumerate}
\item  \textit{Cohesive} if for any pullback diagram of
  connective $\bb{E}_{\infty}$-rings
  \[\begin{tikzcd}
      {R^\prime} & {S^\prime} \\
      R & S \arrow[from=1-2, to=2-2] \arrow[from=2-1, to=2-2] \arrow[from=1-1, to=2-1]
      \arrow[from=1-1, to=1-2]
    \end{tikzcd}\]
  which induces surjections $\pi_{0}R \to \pi_{0}S$ and $\pi_{0}S\p \to \pi_{0}S$, the diagram
  \[\begin{tikzcd}
	{X(R\p)} & {X(S\p)} \\
	{X(R)} & {X(S)}
	\arrow[from=1-1, to=1-2]
	\arrow[from=1-2, to=2-2]
	\arrow[from=1-1, to=2-1]
	\arrow[from=2-1, to=2-2]
\end{tikzcd}\]
is a pullback. We refer to such pullbacks of $\bb{E}_{\infty}$-rings as \textit{small pullbacks}.
   \item \textit{Nilcomplete} if for every connective $\bb{E}_{\infty}$-ring $R$ the natural map
         \[ X(R)\to \flim_{n} X(\tau_{\le n} R)\]
         is an equivalence.
  \end{enumerate}
\end{definition}

\begin{example}
  Let $R$ be a connective $\bb{E}_{\infty}$-ring, then the functor
  \[\rm{Spec}(R): \rm{CAlg}^{\rm{cn}} \to \cl{S} \qquad S \mapsto \rm{Map}_{\rm{CAlg}}(R, S) \]
  commutes with \textit{all} limits, and so in particular it is cohesive and nilcomplete.
\end{example}

\begin{construction}
  Let $X : \rm{CAlg}^{\rm{cn}} \to \cl{S}$ be a cohesive functor, $R \in \rm{CAlg}^{\rm{cn}}$
  and $A \in X(R)$. Then we define a functor
  \[ \rm{X}_{A}^{\blank} :\rm{CAlg}^{\rm{cn}}_{/R} \to \cl{S} \qquad (R\p \to R) \mapsto \rm{fib}_{A}(X(R\p)\to X(R)).\]
  We call $\rm{X}_{A}^{R\p}$ the \textit{space of deformations} of $A$ along $R\p \to R$.
\end{construction}

\begin{remark}
  If $X$ is cohesive and nilcomplete, then for every connective $\bb{E}_{\infty}$-ring $R$
  and every $A\in X(\tau_{\le 0}R)$ we have an equivalence
  \[ X_{A}^{R} \simeq \flim_{n} X_{A}^{\tau_{\le n}R},\]
  i.e.~we can construct lifts to $R$ by inductively lifting against the square zero extensions
  $\tau_{\le n}R \to \tau_{\le n-1}R$. This will be very useful for applications, however we will not need
  nilcompleteness for the theoretical groundwork that makes up the  rest of this chapter.
\end{remark}

\subsection{The tangent complex}\label{sect22}

\begin{definition}
  Let $\cl{C},\cl{D}$ be $\infty$-categories, then a functor $F: \cl{C}\to \cl{D}$ is called:
  \begin{enumerate}
          \item \textit{Reduced} if  it preserves the terminal object.
          \item \textit{Excisive} if it takes pushouts to pullbacks.
  \end{enumerate}
\end{definition}

\begin{proposition}\label{redex}
  Let $X: \rm{CAlg}^{\rm{cn}} \to \cl{S}$ be cohesive and $A \in X(R)$ be an $R$-valued point.
  Then the functor given by the composition
  \[ \rm{Mod}_{R}^{\rm{cn}}\rar{\Omega^{\infty}|_{\rm{Mod}_{R}^{\rm{cn}}}}
    \rm{CAlg}_{/R}^{\rm{cn}} \rar{\rm{X}^{\blank}_{A}} \cl{S} \]
  is reduced and excisive.
\end{proposition}
\begin{proof}
  Clearly, we have $\rm{X}^{R}_{A} \simeq \pt$, so the functor is reduced.
  Since $\rm{X}$ is cohesive and taking fibers commutes with limits,
  the functor $\rm{X}^{\blank}_{A}$ takes small pullbacks to pullbacks.
  Hence, the claim follows from~\cite[][Proposition 1.4.2.13]{ha} by observing that
  $\Omega^{\infty}|_{\rm{Mod}_{R}^{\rm{cn}}}$ sends
  \[\begin{tikzcd}
	M & 0 \\
	0 & {M[1]}
	\arrow[from=1-1, to=1-2]
	\arrow[from=1-2, to=2-2]
	\arrow[from=1-1, to=2-1]
	\arrow[from=2-1, to=2-2]
\end{tikzcd}\]
to the small pullback
\[\begin{tikzcd}
	{R\oplus M} & R \\
	R & {R\oplus M[1]}
	\arrow[from=1-1, to=1-2]
	\arrow["{(\id,0)}", from=1-2, to=2-2]
	\arrow[from=1-1, to=2-1]
	\arrow["{(\id,0)}"', from=2-1, to=2-2],
\end{tikzcd}\]
i.e.~we have $\Omega\rm{X}^{R \oplus M[1]}_{A} \simeq \rm{X}_{A}^{R \oplus M}$ for any $M \in \rm{Mod}_{R}^{\rm{cn}}$.
\end{proof}

\begin{construction}\label{tangent}
  Let $X:\rm{CAlg}^{\rm{cn}} \to \cl{S}$ be cohesive and $A \in X(R)$.
  Then by~\cite[][Proposition 1.4.2.22]{ha} we obtain an essentially unique factorization
  \[\begin{tikzcd}
	& {\rm{Sp}} \\
	{\rm{Mod}_R^{\rm{cn}}} & {\cl{S}}
	\arrow["{X_{A}^{R \oplus \blank}}"', from=2-1, to=2-2]
	\arrow["{\Omega^\infty}", from=1-2, to=2-2]
	\arrow["{T_{X_A}^{\blank}}", dashed, from=2-1, to=1-2]
\end{tikzcd}\]
  where for $M\in \rm{Mod}_{R}^{\rm{cn}}$ the  spectrum $T_{X_{A}}^{M}$ is given by
  the sequence of spaces $\{\rm{X}_{A}^{R \oplus M[n]}\}_{n}$.
 For $M=R$ we call $T^{R}_{X_{A}} =: T_{X_{A}}$ the \textit{tangent complex} of $A$.
\end{construction}

\begin{warning}
  The name tangent complex is a historical convention and somewhat misleading.
  In general, the spectrum $T_{X_{A}}$ is not contained in the full subcategory
  $D(\Z) \subseteq \rm{Sp}$ i.e.~cannot be modeled by a chain complex of abelian groups.
\end{warning}

\begin{lemma}\label{acn}
  For a connective $\bb{E}_{\infty}$-ring $R$ denote by $\rm{Mod}_{R}^{\rm{acn}} \subseteq \rm{Mod}_{R}$ the full
  subcategory spanned by those $R$-modules which are contained in $(\rm{Mod}_{R})_{\geq n}$ for some $n$ and let
  $F: \rm{Mod}_{R}^{\rm{cn}} \to \cl{S}$ be an excisive functor. Then $F$ admits an extension to an excisive
  functor $\rm{Mod}_{R}^{\rm{acn}} \to \cl{S}$ which is unique up to contractible choice.
\end{lemma}
\begin{proof}
  This is~\cite[][Lemma 1.3.2]{dag14}.
\end{proof}

\begin{proposition}\label{structure}
  Let $R$ be a connective $\bb{E}_{\infty}$-ring, $X: \rm{CAlg}^{\rm{cn}} \to \cl{S}$ be cohesive and
  $A \in X(R)$. Then $T_{X_{A}}$ inherits a natural
  $R$-module structure such that for any perfect connective $R$-module $M$ we have
  a natural equivalence $T^{M}_{X_{A}}\simeq T_{X_{A}}\otimes_{R} M$.
\end{proposition}
\begin{proof}
  By Lemma~\ref{acn} we can extend the functor $F=X^{R \oplus \blank}_{X_{A}}: \rm{Mod}_{R}^{\rm{cn}}\to \cl{S}$
  uniquely to an excisive functor $F\p:\rm{Mod}_{R}^{\rm{acn}}\to \cl{S}$. Since $\rm{Mod}_{R}^{\rm{acn}}$
  is stable, $F\p$ is an exact functor. Thus, the restriction $F\p |_{\rm{Mod}_{R}^{\rm{perf}}}$ is also
  exact. Hence, since $\rm{Sp}$ is stable,~\cite[][Proposition 1.4.2.22]{ha} implies that we
  get an essentially unique lift to an exact functor
  $\widetilde{F}: \rm{Mod}_{R}^{\rm{perf}} \to \rm{Sp}$. Finally, applying~\cite[][Proposition 5.5.1.9]{htt}
  we see that $\widetilde{F}$ induces a colimit preserving functor
  $\rm{Ind}(\rm{Mod}_{R}^{\rm{perf}}) \simeq \rm{Mod}_{R}\to \rm{Sp}$, which under the equivalence
  \[ \rm{Fun}^{\rm{L}}(\rm{Mod}_{R}, \rm{Sp})\simeq \rm{Mod}_{R} \qquad G\mapsto G(R)\]
  yields a $R$-module whose underlying spectrum is given by $T_{X_{A}}$.
\end{proof}

\begin{proposition}\label{def}
  Let $X: \rm{CAlg}^{\rm{cn}} \to \cl{S}$ be a cohesive functor and $R^{\eta} \to R$ a square zero extension
  classified by a derivation $R \rar{\eta} M[1]$. Then for each $A \in X(R)$ the space of deformations
  $\rm{X}_{A}^{R^{\eta}}$ is either empty or a torsor under the grouplike $\E_{\infty}$-monoid $\Omega^{\infty}T^{M}_{X_{A}}$.
  Moreover, $\eta$ determines an obstruction class in $\pi_{-1}T^{M}_{X_{A}}$, which vanishes if and only
  if $\rm{X}_{A}^{R^{\eta}}$ is non-empty.
\end{proposition}
\begin{proof}
  Since $X$ is cohesive, applying $\rm{X}_{A}^{\blank}$ to the pullback diagram defining $R^{\eta}$
\[\begin{tikzcd}
	{R^\eta} & R \\
	R & {R\oplus M[1]}
	\arrow[from=1-1, to=1-2]
	\arrow["0", from=1-2, to=2-2]
	\arrow[from=1-1, to=2-1]
	\arrow["{(0,\eta)}"', from=2-1, to=2-2],
\end{tikzcd}\]
we get a pullback of spaces
\[\begin{tikzcd}
	{\rm{X}_A^{R^\eta}} & \pt \\
	\pt & {\rm{X}_A^{R \oplus M[1]}}
	\arrow[from=1-1, to=1-2]
	\arrow["A^{0}",from=1-2, to=2-2]
	\arrow[from=1-1, to=2-1]
	\arrow["A^{\eta}"',from=2-1, to=2-2],
\end{tikzcd}\]
exhibiting $\rm{X}_{A}^{R^{\eta}}$ as the space of paths in $\rm{X}^{R \oplus M[1]}_{A}$ between the points $A^{0}$
and $A^{\eta}$. Hence, it is non-empty if and only if the homotopy class determined by the map
\[ \pt \rar{A^{\eta}} \rm{X}_{A}^{R\oplus M[1]} \simeq \Omega^{\infty}T_{X_{A}}^{M}[1]\]
vanishes. Moreover, in this case $\rm{X}_{A}^{R^{\eta}}$ is a torsor under the loop space based at $A^{0}$,
which is given by
\[ \Omega \rm{X}_{A}^{R\oplus M[1]}\simeq \rm{X}_{A}^{R \oplus M} \simeq \Omega^{\infty}T^{M}_{X_{A}}.\]
\end{proof}

\begin{proposition}\label{bc}
  Let $X: \rm{CAlg}^{\rm{cn}} \to \cl{S}$ be cohesive and $R \to R\p$ a map of connective $\bb{E}_{\infty}$-rings.
  Moreover, let $A \in X(R)$ be a $R$-valued point and denote by $A\p$ the image of $A$
  under the induced map $X(R)\to X(R\p)$.
  Then for every $M \in \rm{Mod}^{\rm{cn}}_{R\p}$ we have a natural map
  \[ T_{X_{A}}^{M} \to T_{X_{A\p}}^{M}\]
  which is an equivalence if the map $\pi_{0}R \to \pi_{0}R\p$ is surjective.
\end{proposition}
\begin{proof}
  Applying $X$ to the pullback of connective $\bb{E}_{\infty}$-rings
\[\begin{tikzcd}
	{R\oplus M} & {R\p \oplus M} \\
	R & R\p
	\arrow[from=2-1, to=2-2]
	\arrow[from=1-2, to=2-2]
	\arrow[from=1-1, to=2-1]
	\arrow[from=1-1, to=1-2]
\end{tikzcd}\]
and taking the fibers over the points $A\in X(R)$ and $A\p\in X(R\p)$ gives a commutative diagram
\[\begin{tikzcd}
	{X^{R\oplus M}_A} & {X_{A\p}^{R\p\oplus M}} \\
	{X(R\oplus M)} & {X(R\p \oplus M)} \\
	{X(R)} & {X(R\p)}
	\arrow[from=3-1, to=3-2]
	\arrow[from=2-2, to=3-2]
	\arrow[from=2-1, to=3-1]
	\arrow[from=2-1, to=2-2]
	\arrow[from=1-1, to=2-1]
	\arrow[from=1-2, to=2-2]
	\arrow[from=1-1, to=1-2]
\end{tikzcd}.\]
The map $X_{A}^{R\oplus M} \to X_{A\p}^{R\p \oplus M}$ is natural in $M$ and thus gives a map of spectra
$T^{M}_{X_{A}}\to T^{M}_{X_{A\p}}$ as claimed. Moreover, if $R \to R\p$ is surjective on $\pi_{0}$, then
the pullback of $\bb{E}_{\infty}$-rings is small. Hence, since $X$ is cohesive, the map $X_{A}^{R\oplus M}\to X_{A\p}^{R\p\oplus M}$
is an equivalence and thus the induced map $T^{M}_{X_{A}}\to T^{M}_{X_{A\p}}$ is as well.
\end{proof}
Notice that, if $R^{\eta} \to R$ is a square zero extension of connective $\bb{E}_{\infty}$-rings
the map $\pi_{0}R^{\eta} \to \pi_{0}R$ is necessarily surjective. Thus, if we are given $A\in X(R)$
and a lift $A^{\eta}\in X(R^{\eta})$, we know that if we have $M\in \rm{Mod}^{\rm{cn}}_{R^{\eta}}$ such that the
$R^{\eta}$-action factors through $R$, then $T^{M}_{X_{A^{\eta}}}$ agrees with $T^{M}_{X_{A}}$. The following
Proposition shows that we can compute the value of $T^{\blank}_{X_{A^{\eta}}}$ on arbitrary connective
$R^{\eta}$-modules in terms of $T^{\blank}_{X_{A}}$.

\begin{proposition}\label{cofib}
  Let $X:\rm{CAlg}^{\rm{cn}} \to \cl{S}$ be cohesive, $R\in \rm{CAlg}^{\rm{cn}}$ and $A\p \in X(R)$. Let
  $R^{\eta} \to R$ be a square zero extension with fiber $M$ such that $A\p$ admits a lift $A\in X(R^{\eta})$.
  Then for any $N\in \rm{Mod}_{R^{\eta}}^{\rm{cn}}$ if we have that
  $T^{M \otimes_{R}(R \otimes_{R^{\eta}}N)}_{X_{A\p}} \simeq T^{N \otimes_{R^{\eta}}R}_{X_{A\p}} \simeq 0$ it follows that
  $T^{N}_{X_{A}} \simeq 0$.
\end{proposition}

\begin{proof}
 Applying the functor $\blank \otimes_{R^{\eta}} N$ to the extension
  \[ M \to R^{\eta} \to R\]
  yields a cofiber sequence
  \[ M \otimes_{R^{\eta}} N \to N \to N \otimes_{R^{\eta}} R\]
  of connective $R^{\eta}$-modules. Now the functor $T^{\blank}_{X_{A}}: \rm{Mod}_{R^{\eta}}^{\rm{cn}} \to \rm{Sp}$
  is excisive, hence we get a fiber sequence of spectra
  \[T^{M\otimes_{R^{\eta}}N}_{X_{A}} \to T^{N}_{X_{A}}\to T^{N \otimes_{R^{\eta}}R}_{X_{A}}.\]
  Since the extension $R^{\eta} \to R$ is square zero,
  the action of $R^{\eta}$ on $M$ factors through $R$, i.e.~we have that
  \[ M \otimes_{R^{\eta}}N \simeq (M \otimes_{R}R) \otimes_{R^{\eta}} N 
  \simeq M \otimes_{R} (R \otimes_{R^{\eta}} N).\]
  Applying Proposition~\ref{bc} we see that
  \[ T^{M \otimes_{R}(R \otimes_{R^{\eta}}N)}_{X_{A}}\simeq  T^{M \otimes_{R}(R \otimes_{R^{\eta}}N)}_{X_{A\p}} \simeq 0\]
  and similarly
  \[ T^{N \otimes_{R^{\eta}}R}_{X_{A}}\simeq T^{N \otimes_{R^{\eta}}R}_{X_{A\p}} \simeq 0,\]
  which proves the claim.
\end{proof}

\begin{remark}
  Note that in the setting of Proposition~\ref{cofib}, if $M$ and $N \otimes_{R^{\eta}}R$ are perfect $R$-modules,
  Proposition~\ref{structure} implies that it suffices to assume that $T_{X_{A\p}} \simeq 0$.
  Moreover, as part of the proof we have seen that every connective $R^{\eta}$-Module $N$ sits
  in a cofiber sequence
  \[ M \otimes_{R} (R \otimes_{R^{\eta}} N) \to N \to R \otimes_{R^{\eta}} N.\]
  If we think of $N$ as a lift of $R \otimes_{R^{\eta}} N$ along the square zero extension $R^{\eta}\to R$,
  this is part of a description of the deformation theory of connective modules. The complete
  description may be deduced from Proposition~\ref{Mod}.
\end{remark}



As an example we now summarize the well-known deformation theory of $\bb{E}_\infty$-rings in our language.
Later we will use this description to analyze the deformation theory of dualizable coalgebras.

\begin{example}~\label{spec}
  Let $R$ be any $\bb{E}_{\infty}$-ring.
  The composition
  \[ \rm{Mod}_{R}\rar{\Omega^{\infty}} \rm{CAlg}_{/R} \rar{\rm{Map}(R, \blank)} \cl{S}\]
  is accessible and commutes with limits. Thus, since $\rm{Mod}_{R}$ is presentable, the adjoint functor
  theorem implies that it is corepresented by an $R$-Module $L_{R}$ called the \textit{cotangent complex}
  of $R$. Now the functor
  \[ X=\rm{Spec}(R): \rm{CAlg}^{\rm{cn}}\to \cl{S} \qquad S \mapsto \rm{Map}_{\rm{CAlg}}(R, S)\]
  is clearly cohesive.
  Moreover, for any $(\varphi:R \to S)\in \rm{Spec}(R)(S)$ and $M \in \rm{Mod}_{S}^{\rm{cn}}$ we get that
  \begin{align*}
    &\rm{fib}_{\varphi}(\rm{Map}_{\rm{CAlg}}(R, S \oplus M) \to \rm{Map}_{\rm{CAlg}}(R,S))\\
    \simeq ~ &\rm{Map}_{\rm{CAlg}_{/S}}(R, S \oplus M)\\
    \simeq ~ &\rm{Map}_{\rm{CAlg}_{/R}}(R, R \oplus \varphi_{\pt}M)\\
    \simeq  ~ &\rm{Map}_{R}(L_{R}, \varphi_{\pt}M)\\
    \simeq ~ &\rm{Map}_{S}(\varphi^{\pt}L_{R}, M).
  \end{align*}
  Hence, for each $M$ we have an equivalence
  \[ \rm{map}_{S}(\varphi^{\pt}L_{R}, M) \simeq T_{X_{\varphi}}^{M}.\]
Explicitly, this tells us that
  the space of lifts in the diagram
  \[\begin{tikzcd}
	&&& {S\oplus M} \\
	{} && R & S
	\arrow["\varphi", from=2-3, to=2-4]
	\arrow[from=1-4, to=2-4]
	\arrow[dashed, from=2-3, to=1-4]
\end{tikzcd}\]
is naturally identified with
$\Omega^{\infty}T_{X_{\varphi}}^{M}=\rm{Map}_{S}(\varphi^{\pt}L_{R}, M)$. Moreover, if $L_{R} \simeq 0$, then $R$ admits
unique lifts against \textit{arbitrary} square zero extensions. In the case $S=R$ and $\varphi=\id$
$\rm{Map}_{\rm{CAlg}_{/R}}(R, R \oplus M) = \rm{Map}_{R}(L_{R}, M)$ is also called the space of
\textit{derivations} $R\to M$, which in the discrete case can be explicitly described via
additive maps satisfying the Leibniz rule. Although the existence of a cotangent complex
for coalgebras remains unclear, we will show that there is a coalgebraic notion
of derivations which play a similar role in the deformation theory.
\end{example}

\begin{example}\label{counterex}
  Proposition~\ref{Mod} implies that the functor
  \[ X:\rm{CAlg}^{\rm{cn}} \to \cl{S} \qquad R \mapsto (\rm{CAlg}_{R}^{\rm{cn}})^{\Delta^{0}}\]
  is cohesive. Moreover, it follows from~\cite[][Proposition 7.4.2.5]{ha} that for every $S \in X(R)$
  there exists an $S$-module $L_{S/R}$ called the \textit{relative cotangent complex}, together with,
  for every connective $S$-module $M$ a natural equivalence
  \[ \rm{Map}_{S}(L_{S/R}, M[1] \otimes_R S) \rar{\sim} \rm{fib}_{S}(X(R\oplus M) \to X(R)).\]
  Thus, the problem of lifting $S$ to a $R\oplus M$-algebra $\widetilde{R}$ is equivalent to
  finding a map of $R$-algebras fitting into the diagram
  \[\begin{tikzcd}
	& {S \oplus (S \otimes_R M[1])} \\
	S & S
	\arrow["\id", from=2-1, to=2-2]
	\arrow[from=1-2, to=2-2]
	\arrow[dashed, from=2-1, to=1-2].
\end{tikzcd}\]
Meaning we get an equivalence
\[ \rm{map}_{S}(L_{S/R}, M[1]\otimes_{R} S) \simeq T^{M}_{X_{S}}\]
\end{example}

\section{Deformation theory of coalgebras}
We now apply the machinery reviewed in the previous section to study deformation theoretic
questions about coalgebras in the category of spectra. We first prove that the functors which
assign to a connective $\E_\infty$-ring $R$ the categories of connective $R$-modules and $p$-complete
connective $R$-modules respectively, are cohesive an nilcomplete and that this implies the same for
coalgebras in those categories.
We then introduce \textit{formally \'etale coalgebras} in an arbitrary cohesive moduli problem
\[\cC_\blank: \CAlg^{\cn} \to \CAlg(\rm{Pr^L})\]
and show that the space of lifts of a formally
\'etale coalgebra $A\in \cCAlg(\cC_R)$ along any square zero extension $R^\eta \to R$ is contractible.
Moreover, we show that the assignment of $A$ to its essentially unique
lift $A^{\eta} \in \rm{cCAlg}(\cC_{R^{\eta}})$ refines to a fully faithful
functor $\rm{cCAlg}(\cC_R)^{\rm{f\acute{e}t}} \to \rm{cCAlg}(\cC_{R^{\eta}})$.

\subsection{Moduli of coalgebras}

We now prove that moduli of coalgebras are cohesive and nilcomplete,
and hence we can use the tangent complex machinery discussed in Section~\ref{sect22}. We deduce
this from the fact that the categories of connective modules over a connective $\bb{E}_\infty$-ring
commute with the corresponding limits and the equivalences are strong monoidal.

\begin{lemma}\label{edescent}
Let $\cl{C}_{\blank}: \rm{CAlg}^{\rm{cn}} \to \CAlg(\rm{Pr^L})$
be cohesive or nilcomplete. Then the functor
$\rm{cCAlg}(\cl{C}_\blank): \rm{CAlg}^{\rm{cn}} \to \rm{Pr^L}$ is
is also cohesive or nilcomplete, respectively.
\end{lemma}
\begin{proof}
Clear, since by Proposition~\ref{calg} and Lemma~\ref{limits} the functor
$\cCAlg(\blank)$ commutes with limits.
\end{proof}

\begin{theorem}[Lurie]
  Suppose we have a pullback of connective $\bb{E}_{\infty}$-rings
  \[\begin{tikzcd}
      {R^\prime} & {S^\prime} \\
      R & S \arrow[from=1-1, to=1-2] \arrow[from=1-2, to=2-2]
      \arrow[from=1-1, to=2-1] \arrow[from=2-1, to=2-2]
    \end{tikzcd}\]
  such that one of the maps $\pi_{0}R \to \pi_{0}S$, $\pi_{0}S\p \to \pi_{0}S$ is surjective.
  Then the natural map
  \[ \rm{Mod}^{\rm{cn}}_{R\p} \to \rm{Mod}_{R}^{\rm{cn}} \times_{\rm{Mod}_{S}^{\rm{cn}}} \rm{Mod}_{S\p}^{\rm{cn}}\]
  is an equivalences of categories with inverse taking a point in the pullback $(M,N,h)$, consisting of
  $M \in\rm{Mod}_{R}^{\rm{cn}}, N \in \rm{Mod}_{S\p}^{\rm{cn}}$ and a homotopy $h: M\otimes_{R}S\simeq N \otimes_{S\p}S$, to
  $M \times_{M \otimes_{R}S} N$ with the induced $R\p$-module structure.
\end{theorem}
\begin{proof}
\cite[][Theorem 16.2.0.2.]{sag}
\end{proof}

\begin{corollary}\label{Mod}
    The functor
    \[ \CAlg^{\rm{cn}} \to \CAlg(\rm{Pr^L}) \quad R \mapsto \Mod_R^{\rm{cn}} \]
    is cohesive.
\end{corollary}

This shows that moduli of coalgebras are cohesive in the sense of Definition~\ref{cohesive}.
To prove that they are also nilcomplete we need the following technical Lemma.

\begin{lemma}\label{conn}
  Let $\dots \to  E_{2} \to E_{1} \to E_{0}$ be a diagram of spectra. Suppose we are given $L \ge 0$ such
  that for all $\ell\p\ge\ell \ge L$ the map $E_{\ell\p}\to E_{\ell}$ is $m$-connective. Then for any $\ell \ge L$ the map
  \[ \flim_{n} E_{n}\to E_{\ell}\]
  is $m-1$-connective.
\end{lemma}

\begin{proof}
  Writing $F_{\ell\p,\ell}= \rm{fib}(E_{\ell\p}\to E_{\ell})$ and $F_{\ell}= \rm{fib}(\lim_{n}E_{n}\to E_{\ell})$ we want
  to show that $F_{\ell}$ is $m$-connective. Indeed, since limits are exact, we have that
  \[F_{\ell} \simeq \lim_{\ell\p >\ell}F_{\ell\p,\ell} \simeq \rm{fib}\left( \prod_{\ell\p >\ell} F_{\ell\p,\ell}\to \prod_{\ell\p>\ell} F_{\ell\p -1, \ell}\right ).\]
  Thus, since $\rm{Sp}_{\geq m}$ is closed under products and the fiber of a map of $m$-connective spectra
  is $(m-1)$-connective, the claim follows.
\end{proof}

\begin{proposition}\label{nilmod}
The functor
\[ \CAlg^{\rm{cn}} \to \CAlg(\rm{Pr^L}) \quad R \mapsto \rm{Mod}_R^{\rm{cn}}\]
is nilcomplete.
\end{proposition}
\begin{proof}
  Let $R$ be a connective $\bb{E}_{\infty}$-ring. We need to show that the functor
  \[\rm{Mod}_{R}^{\rm{cn}} \to \flim_{n} \rm{Mod}_{\tau_{\le n}R}^{\rm{cn}} \quad M \mapsto M \otimes_{R} \tau_{\le n}R\]
  is an equivalence of categories.
  Write $R_{n}:= \tau_{\leq n}R$. The functor admits a right adjoint which
  takes $(M_{n})\in \flim_{n} \rm{Mod}_{R_{n}}^{\rm{cn}}$ to the limit $\lim_{n} M_{n}$ which inherits
  a natural action by $\lim_{n} R_{n}\simeq R$. Now let $N\in \rm{Mod}_{R}$. Since taking limits is exact,
  the counit of the adjunction sits in a fiber sequence
  \[\lim_{n} \rm{fib}(N \to N \otimes_{R} R_{n}) \to N \rar{\eta} \lim_{n} (N \otimes_{R} R_{n}).\]
  where we compute for the left hand term that
  \[ \rm{fib}(N \to N \otimes_{R} R_{n}) \simeq \rm{fib}(N \otimes_{R} R \to N \otimes_{R} R_{n}) \simeq N \otimes_{R} \rm{fib}(R \to R_{n}).\]
  Now, since $R_{n}= \tau_{\leq n} R$, the connectivity of $\rm{fib}(R\to R_{n})$ increases with $n$. Hence,
  since $R$ and $N$ are connective, so does the connectivity of the tensor product
  $N \otimes_{R} \rm{fib}(R \to R_{n})$ which implies that $\flim_{n}N \otimes_{R} \rm{fib}(R \to R_{n}) \simeq 0$. Thus,
  the counit $N \to \lim_{n}(N \otimes_{R} R_{n})$ is an equivalence. \\
  Now let $(M_{n}) \in \flim_{n}\rm{Mod}_{R_{n}}^{\rm{cn}}$ and write $M= \lim_{n} M_{n}$.
  We need to show that the natural map
  \[ \eps_{k}:M \otimes_{R} R_{k} \to R_{k}\]
  is an equivalence for each $k$. We do this by showing that it is $m$-connective for any $m\ge 0$.
  Indeed, for any such $m$ there exists an integer $L$ such that for all
  $\ell \geq \ell\p > L$ the natural map $R_{\ell} \to R_{\ell\p}$ is $m$-connective. Since
  $M_{\ell\p}\simeq M_{\ell}\otimes_{R_{\ell}} R_{\ell\p}$  we have a fiber sequence
  \[ M_{\ell}\otimes_{R_{\ell}} (\rm{fib}(R_{\ell}\to R_{\ell\p})) \to M_{\ell} \to M_{\ell\p}.\]
  Hence, since $\rm{fib}(R_{\ell}\to R_{\ell\p})$ is $m$-connective and $R_{\ell}$ and $M_{\ell}$ are connective, the tensor
  product $M_{\ell}\otimes_{R_{\ell}} \rm{fib}(R_{\ell}\to R_{\ell\p})$ is $m$-connective as well. Thus, for fixed $m$ and $k$ we
  may apply Lemma~\ref{conn} to obtain  $\ell>k$ such that the maps $M \to M_{\ell}$ and $R\to R_{\ell}$ are
  both $m$-connective. Finally, the map
  \[ \eps_{k}: M\otimes_{R}R_{k} \to M_{\ell}\otimes_{R_{\ell}}R_{k} \simeq M_{k}\]
  is given by the colimit of the induced map between the bar resolutions
\[\begin{tikzcd}
	\vdots & \vdots \\
	{M\otimes R \otimes R_k} & {M_\ell\otimes R_\ell \otimes R_k} \\
	{M \otimes R_k} & {M _\ell \otimes R_k}
	\arrow[from=1-1, to=2-1]
	\arrow[from=1-2, to=2-2]
	\arrow[shift left=2, from=2-1, to=3-1]
	\arrow[shift left=2, from=2-2, to=3-2]
	\arrow[from=3-1, to=3-2]
	\arrow[from=2-1, to=2-2]
	\arrow[shift right=3, from=1-1, to=2-1]
	\arrow[shift left=3, from=1-1, to=2-1]
	\arrow[shift right=3, from=1-2, to=2-2]
	\arrow[shift left=3, from=1-2, to=2-2]
	\arrow[shift right=2, from=2-1, to=3-1]
	\arrow[shift right=2, from=2-2, to=3-2]
\end{tikzcd}.\]
Denote by $F_{n}$ the fiber of the map $M \otimes R^{\otimes n}\otimes R_{k} \to M_{\ell}\otimes R_{\ell}^{\otimes n}\otimes R_{k}$. Since the tensor product
of $m$-connective maps is again $m$-connective, the fiber $F_{n}$ is $m$-connective. Thus, by exactness
of colimits, we obtain a fiber sequence
\[\colim F_{n} \to M \otimes_{R} R_{k} \rar{\eps_{k}} M_{k}\]
and finally, since taking colimits preserves connectivity, this shows that the map
$\eps_{k}$ is $m$-connective. Since $m$ was arbitrary, the map $\eps_{k}$ is in fact an equivalence
which completes the proof.
\end{proof}

\begin{corollary}\label{cohesive}
The functor $\cCAlg(\Mod_{\blank}^{\rm{cn}}):\CAlg^{\rm{cn}} \to \rm{Pr^L}$ is cohesive and nilcomplete.
\end{corollary}

\begin{proof}
Apply Lemma~\ref{edescent} to Proposition~\ref{nilmod} and Corollary~\ref{Mod}.
\end{proof}

\subsection{$p$-complete moduli}

Throughout this section fix a prime $p$.
Let $R$ be an $\bb{E}_{\infty}$-ring. Recall that a module $M \in \rm{Mod}_{R}$ is called
$p$-\textit{complete} if the limit
\[ \lim \left(\dots \rar{\cdot p} M \rar{\cdot p}M \right)\]
vanishes. We denote the full subcategory spanned by the $p$-complete modules 
by $\rm{Mod}_{R}^{\wedge}$.
The inclusion $\rm{Mod}_{R}^{\wedge} \to \rm{Mod_{R}}$ admits a left adjoint which 
takes a module $M$ to its \textit{$p$-completion} given by the limit
\[ M^\wedge_p:=\lim \left( \dots \to M/p^{2} \to M/p \right).\]
In fact, $M$ is $p$-complete if and only if the natural map $M \to \lim M/p^{n}$ is an equivalence.
The spectrum $M^\wedge_p$ inherits a natural $R^{\wedge}_{p}$-module structure,
and $p$-completion induces an equivalence of categories
\[\rm{Mod}_{R}^{\wedge} \simeq \rm{Mod}_{R^{\wedge}_{p}}^{\wedge}.\]
The tensor product of $p$-complete modules is in general not $p$-complete. However, the
category $(\rm{Mod}_{R})_{p}^{\wedge}$ admits a presentably symmetric monoidal structure
given by the formula
 \[ M \otimes_{(\rm{Mod}_{R})_{p}^{\wedge}} N := ( M \otimes N )^{\wedge}_{p}.\]
 With this monoidal structure the $p$-completion functor $\rm{Mod}_{R}\to (\rm{Mod}_{R})_{p}^{\wedge}$
 is monoidal, while the inclusion is only lax monoidal. 

\begin{lemma}
    Let $R_{\blank}:I \to \CAlg^{\rm{cn}}$ be a diagram of connective $\bb{E}_\infty$ rings with limit $R$, such
    that the natural functor
    \[ F:\Mod_R \to \flim_{i\in I} \Mod_{R_i}\]
    is an equivalence of categories. Then the induced functor
    \[ F^\wedge:\Mod_R^\wedge \to \flim_{i \in I} \Mod_{R_i}^\wedge\]
    is also an equivalence.
\end{lemma}
\begin{proof}
    This follows since, for a point $(M_i) \in \flim_i \Mod_{R_i}$, the underlying spectrum
    of $F^{-1}(M)$ is computed as the limit of spectra $\lim_i M \otimes R_i$ and $p$-completion
    commutes with limits and is monoidal.
\end{proof}

 \begin{corollary}\label{pnil}
 The functor 
 \[ \CAlg^{\rm{cn}} \to \rm{CAlg}(\rm{Pr^L}) \quad R \mapsto \Mod_R^{\cn,\wedge}\]
 is cohesive and nilcomplete.
 \end{corollary}

\begin{proposition}\label{pcomp}
 Let $k$ be a perfect $\F_p$-algebra and denote by $W_0(k)$ the $p$-typical Witt-vectors of $k$. 
 The functor
\[ \rm{Mod}_{W_0(k)}^\wedge \to \lim_n \rm{Mod}_{W_0(k)/p^n} \quad N \mapsto N 
\otimes_{W_0(k)} W_0(k)/p^n \]
is a strong monoidal equivalence.
\end{proposition}
\begin{proof}
  Let us first reduce to the case $k= \F_p$ and $W_0(R)= \Z_p$.
  Indeed, suppose we have proven that case, then we have equivalences
  \begin{align*}
      \rm{Mod}_{W_0(k)}^\wedge &\simeq \rm{Mod}_{W_0(k)}(\rm{Mod}_{\Z_p}^\wedge)\\
      &\simeq \rm{Mod}_{W_0(k)}(\lim_n \rm{Mod}_{\Z/p^n})\\
      &\simeq \lim_n \rm{Mod}_{W_0(k)/p^n}(\rm{Mod}_{\Z/p^n})\\
      &\simeq \lim_n \rm{Mod}_{W_0(k)/p^n},
  \end{align*}
  and so we are done. Thus, we assume $k= \F_p$ in the following.
  The functor admits a right adjoint which takes $(M_{n})\in \flim_{n}\rm{Mod}_{\Z/p^{n}}$ to the limit
  $\lim_{n}M_{n}$ taken in the category of $\Z_{p}$-modules. Since $p$-complete modules are closed under
  limits, the essential image of this functor is contained in $\rm{Mod}_{\Z_{p}}^{\wedge}$. Moreover,
  if $M\in \rm{Mod}_{\Z_{p}}^{\wedge}$, then we have that
  \[ \flim_{n}(M \otimes_{\Z_{p}} \Z/p^{n}) \simeq \flim_{n} M/p^{n} \simeq M^{\wedge}_{p}\simeq M.\]
  Hence, the counit of the adjunction is an equivalence on $p$-complete modules.
  Conversely, given $(N_{k})\in \flim_{k}\rm{Mod}_{\Z/p^{k}}$ write $N= \lim_{k}N$. We want
  to show that, for every $n$ the natural map
  \[ N \otimes_{\Z_{p}} \Z/p^{n}\rar{\sim}N_{n}\]
  is an equivalence. Since $N \otimes_{\Z_{p}}Z/p^{n}\simeq N/p^{n}$ and limits are exact, we have an equivalence
  \[N \otimes_{\Z_{p}}\Z/p^{n}\simeq \lim_{k >n}(N_{k}\otimes_{\Z_{p}}\Z/p^{n}).\]
  Thus, the unit of the adjunction may be written as
  \[ \lim_{k>n}(N_{k} \otimes_{\Z_{p}}\Z/p^{n}) \to \lim_{k>n}(N_{k}\otimes_{\Z/p^{k}}\Z/p^{n})\simeq N_{n}\]
  and so has fiber given by
  \[ F_{n}:=\lim_{k>n}\left(N_{k}\otimes_{\Z/p^{k}}\rm{fib}(\Z/p^{k}\otimes_{\Z_{p}}\Z/p^{n}\to \Z/p^{n}) \right).\]
  Now we compute the fiber of $\Z/p^{k}\otimes_{\Z_{p}}\Z/p^{n}\to \Z/p^{n}$ as the module
  \[ \rm{Tor}^{\Z_{p}}(\Z/p^{k}, \Z/p^{n})[1]\simeq \Z/p^{n}[1].\]
  The reduction map $\Z/p^{k}\to \Z/p^{k-1}$ is induced by the map of projective resolutions
\[\begin{tikzcd}
	{\Z_p} & {\Z_p} \\
	{\Z_p} & {\Z_p}
	\arrow["{\cdot p^k}", from=1-1, to=1-2]
	\arrow["\id", from=1-2, to=2-2]
	\arrow["{\cdot p}"', from=1-1, to=2-1]
	\arrow["{\cdot p^{k-1}}"', from=2-1, to=2-2],
\end{tikzcd}\]
hence, on Tor it induces the multiplication by $p$ map
\[ \Z/p^{n}=\rm{Tor}^{\Z_{p}}(\Z/p^{k}, \Z/p^{n})\rar{\cdot p} \rm{Tor}^{\Z_{p}}(\Z/p^{k-1}, \Z/p^{n}) =\Z/p^{n}.\]
Thus, if we have $k\p > k > n$ such that $k\p -k > n$, the transition map
\[ F_{k\p}=N_{k\p} \otimes \rm{Tor}^{\Z_{p}}(\Z/p^{k}, \Z/p^{n})\to N_{k} \otimes \rm{Tor}^{\Z_{p}}(\Z/p^{k-1}, \Z/p^{n})= F_{k}\]
vanishes since the Tor-groups are $p^{n}$-torsion. Choosing a cofinal subset $S\subseteq \bb{N}_{>n}$ such that
$\abs{k\p -k}> n$ for any distinct $k\p,k\in S$, we see that
\[ \lim_{k>n} F_{k}\simeq \lim_{k\in S} F_{k} \simeq 0 \]
vanishes. Thus, since limits are exact, the map $N \otimes_{\Z_{p}} \Z/p^{n}\rar{\sim}N_{n}$ is an equivalence.\\
To see that the functor $\rm{Mod}_{\Z_{p}}^{\wedge} \to \flim_n \rm{Mod}_{\Z/p^{n}}$ is strong monoidal,
we observe that since cofibers and limits are exact, we have for each $n$ equivalences
\begin{align*}
  (M \otimes_{\Z_{p}} N)^{\wedge}_{p} \otimes_{\Z_{p}}\Z/p^{n} &\simeq \lim_{k}(M/p^{k} \otimes_{\Z_{p}}N/p^{k})/p^{n}\\
                                              &\simeq \lim_{k}\left((M/p^{n} \otimes_{\Z_{p}} N/p^{n})\otimes_{Z_{p}}\Z/p^{k}\right) \\
  &\simeq ((N\otimes_{\Z_{p}}\Z/p^{n}) \otimes_{\Z_{p}} (M \otimes_{\Z_{p}}\Z/p^{n}))^{\wedge}_{p}.
\end{align*}
This proves the claim.
\end{proof}

\begin{corollary}\label{pcomp1}
  We have an equivalence of categories
  \[ \rm{cCAlg}(\Mod_{W_0(k)}^{\wedge} \rar{\sim} \flim_{n} \rm{cCAlg}(\Mod_{W_0(k)/p^{n}})
  \quad A \mapsto (A\otimes_{W(R)} W(R)/p^{n})\]
  with inverse taking a system of coalgebras $(B_{n})$ to the limit $\lim_{n}B_{n}$ computed in the
  category of $p$-complete $W_0(k)$-modules.
\end{corollary}
\begin{proof}
This follows from Proposition~\ref{pcomp}, arguing as in the proof of Proposition~\ref{Mod}.
\end{proof}

We also observe that the tangent complex is not affected by $p$-completion of coalgebras.

  \begin{lemma}\label{pcomparison}
    Write $\cl{X}(\blank)=\rm{cCAlg}(\Mod^{\rm{cn}}_{\blank})$ and $\cl{Y}(\blank)=
    \rm{cCAlg}(\Mod^{\rm{cn}~\wedge}_{\blank})$. Then the $p$-completion map $f:\cl{X}\to \cl{Y}$
    induces an equivalence
    \[ T^{M}_{(\cl{X}^{\Delta^{n}})_{\xi}} \to  T^{M}_{(\cl{Y}^{\Delta^{n}})_{f(\xi)}}\]
        for every $\F_{p}$-module $M$, $n\in \bb{N}$ and $\xi \in \cl{X}(\F_{p})^{\Delta^{n}}$.
  \end{lemma}
  \begin{proof}
    Since $\F_{p}$-algebra $R$ is $p$-complete, the $p$-completion functor gives an equivalence
    $\rm{Mod}_{R}\rar{\sim} \rm{Mod}_{R}^{\wedge}$, since multiplication by some power of $p$
    is nullhomotopic over $\F_{p}$. In particular, this applies to the split square zero
    extension $\F_{p}\oplus M$ for any $M \in \rm{Mod}_{\F_{p}}$ and so the natural map
    $\cl{X}(\F_{p}\oplus M) \to \cl{Y}(\F_{p}\oplus M)$ is an equivalence as well.
    Consequently, we also obtain natural equivalences between the fibers
    \[ (\cl{X}^{\Delta^{n}})_{\xi}^{\F_{p}\oplus M} \to  (\cl{Y}^{\Delta^{n}})_{f(\xi_)}^{\F_{p}\oplus M},\]
    which induces the equivalence of spectra
    \[ T^{M}_{(\cl{X}^{\Delta^{n}})_{\xi}} \to  T^{M}_{(\cl{Y}^{\Delta^{n}})_{f(\xi)}}\]
      as claimed.
  \end{proof}

\subsection{Formally \'etale coalgebras}

Throughout this section, we fix a cohesive functor
\[ \cl{C}_{\blank}: \rm{CAlg}^{\rm{cn}}\to \CAlg(\rm{Pr^L}).\]
Then by Lemma~\ref{limits} the functor
\[ \cCAlg(\cl{C}_\blank): \rm{CAlg}^{\rm{cn}} \to \rm{Pr^L}\]
is also cohesive and for any map $f: R \to S$ of connective $\bb{E}_\infty$ rings we have a coalgebraic adjunction
\[ f^\pt: \cCAlg(\cl{C}_{R}) \leftrightarrows \cCAlg(\cl{C}_S):f_\pt.\]
We refer to $f^\pt$ as base change and to $f_\pt$ as Weil restriction along $f$.

\begin{construction}\label{univdef}
Let $R$ be a connective $\bb{E}_\infty$-ring and $M$ a connective $R$-module. Denote by $e: R \to R\oplus M$
the 0-section of the split square zero extension.
    For any $A\in \rm{cCAlg}(\cl{C}_R)$ we define the \textit{universal $M$-deformation coalgebra}
    of $A$ as the Weil restriction
    \[ \Omega^{\infty}_{A}M:= e_{\pt}e^{\pt} A \in \rm{cCAlg}(\cl{C}_R).\]
    This naturally receives a unit map $\varepsilon:A \to \Omega_A^\infty M$.
\end{construction}

\begin{construction}\label{pi}
Suppose we are given an adjunction
\[\begin{tikzcd}
	{f^\pt:\cC} & {\cD:f_\pt}
	\arrow[""{name=0, anchor=center, inner sep=0}, shift left=2, from=1-1, to=1-2]
	\arrow[""{name=1, anchor=center, inner sep=0}, shift left=2, from=1-2, to=1-1]
	\arrow["\dashv"{anchor=center, rotate=-90}, draw=none, from=0, to=1]
\end{tikzcd}\]
 such that $f^\pt$ admits a retract $g^{\pt}:\cl{D}\to \cl{C}$. Consider the natural transformation
  \[\pi: f_{\pt}f^{\pt} \rar{\sim} g^{\pt} f^{\pt}f_{\pt}f^{\pt}
  \to  g^{\pt}f^{\pt} \rar{\sim} \id_{\cl{C}}\]
  defined as the whiskering of the counit $\eps:f^{\pt}f_{\pt} \to \id$ as in the diagram
\[\begin{tikzcd}
	{\cl{C}} & {\cl{D}} & {\cl{D}} & {\cl{C}}
	\arrow[""{name=0, anchor=center, inner sep=0}, "{f^\pt f_\pt}", curve={height=-12pt}, from=1-2, to=1-3]
	\arrow[""{name=1, anchor=center, inner sep=0}, "\id"', curve={height=12pt}, from=1-2, to=1-3]
	\arrow["{f^\pt}", from=1-1, to=1-2]
	\arrow["{g^{\pt}}", from=1-3, to=1-4]
	\arrow[shorten <=3pt, shorten >=3pt, Rightarrow, from=0, to=1].
\end{tikzcd}\]
Unraveling the definition we see that for each $B,A \in \cl{C}$ the composition
\[ \rm{Map}_{\cl{C}}(B, f_{\pt}f^{\pt} A) \rar{\sim}\rm{Map}_{\cl{D}}(f^{\pt}B, f^{\pt} A)
 \rar{g^{\pt}} \rm{Map}_{\cl{C}}(B, A)\]
takes $\psi: B \to f_{\pt}f^{\pt}A$ to the composite $\pi_{A} \circ \psi$. Thus, for each $\varphi:B \to A$ we have an
equivalence between the fiber
\[ \rm{fib}_{\varphi}\left(\rm{Map}_{\cl{D}}(f^{\pt}B, f^{\pt}A) \rar{g^{\pt}}
    \rm{Map}_{\cl{C}}(B, A) \right)\]
and the mapping space
\[ \rm{Map}_{\cl{C}_{/A}}((B\rar{\varphi} A), (f_{\pt}f^{\pt}A \rar{\eta_{A}} A)).\]
\end{construction}

\begin{lemma}\label{prlspectra}
Let $\cC =\lim_i\cC_i$ be a limit diagram in $\CAlg(\rm{Pr^L})$ and suppose we are
given a map $f: \cD\to \cC$ with projections $f_i: \cD \to \cC_i$. Then 
for any $A\in \cCAlg(\cl{\cD})$ we have a natural equivalence
\[ f_\pt f^\pt A \simeq \lim_i (f_i)_\pt (f_i)^\pt A \in \cCAlg(\cD). \]
\end{lemma}
\begin{proof}
Unravelling the adjunction and using the Yoneda lemma, this is equivalent to 
the claim that for any $B\in \cCAlg(\cD)$ the natural map
\[ \Map_{\cCAlg(\cC)}(f^\pt B, f^\pt A) \to \lim_i \Map_{\cCAlg(\cC_i)}(f_i^\pt B, f_i^\pt A) \]
is an equivalence, which is precisely the formula for mapping spaces
in the limit $\cC=\lim_i\cC_i$.
\end{proof}

\begin{proposition}\label{specref}
    Let $R$ be a connective $\bb{E}_\infty$-ring and $M\in \Mod_R^{\cn}$. Then for any 
    $A\in \cCAlg(\cC_R)$ we have a natural pullback diagram
    \[\begin{tikzcd}
	{\Omega^\infty_AM} & A \\
	A & {\Omega^\infty_A M[1]}
	\arrow[from=1-1, to=1-2]
	\arrow[from=1-1, to=2-1]
	\arrow[from=1-2, to=2-2]
	\arrow[from=2-1, to=2-2].
\end{tikzcd}\]
\end{proposition}
\begin{proof}
Indeed, since $\cCAlg(\cC_{\blank})$ is cohesive we have a pullback diagram in $\rm{Pr^L}$
of the form
    \[\begin{tikzcd}
	{\cCAlg(\cC_{R\oplus M})} & {\cCAlg(\cC_{R})} \\
	{\cCAlg(\cC_R)} & {\cCAlg(\cC_{R\oplus M[1]})}
	\arrow[from=1-1, to=1-2]
	\arrow[from=1-1, to=2-1]
	\arrow[from=1-2, to=2-2]
	\arrow[from=2-1, to=2-2],
\end{tikzcd}\]
where the maps are base change along the projection $p: R\oplus M \to R$ and 0-section
$s:R\to R\oplus \Sigma M$ respectively. Base change along the 0-section $e: R\to R\oplus M$
gives a map into the pullback $e^\pt: \cCAlg(\cC_R) \to \cCAlg(\cC_{R\oplus M})$. Unwrapping
the definitions and using that $p^\pt e^\pt =\id$, the claim follows by applying
Lemma~\ref{prlspectra}.
\end{proof}

\begin{definition}\label{formalet}
    Let $R$ be a connective $\bb{E}_\infty$-ring and $A\in \cCAlg(\cC_R)$. 
    By Proposition~\ref{specref}, the assignment
    $M \mapsto \Omega^\infty_A M $ canonically refines to a functor
    \[ L_A: \Mod_R \to \Sp(\cCAlg(\cC_R)_{/A}),\]
    such that $\Omega^\infty L_A(M)\simeq \Omega^\infty_A M$. We call $A$ \textit{formally \'etale} 
    if $L_A\simeq 0$ and denote by $\cCAlg(\cC_R)^{\fet}$ the full subcategory spanned
    by the formally \'etale coalgebras.
\end{definition}

\begin{remark}
    By construction of the functor $L_A$, a coalgebra $A \in \cCAlg(\cC_R)$ 
    is formally \'etale if and only if for any $M\in \Mod_{R}^{\cn}$ either of the natural maps
    \[ A \rar{\eps} \Omega^\infty_A M \rar{\pi} A\]
   is an equivalence. In fact, since the composition $\pi \circ \eps$ is always homotopic to the 
   identity, it suffices to show that there exists some isomorphism $\Omega^\infty_A M \simeq A$.
   This is how we verify the condition in practice. 
\end{remark}

\begin{remark}
    We do not know whether the functor $L_A$ commutes with limits or colimits. Hence, all we may
    deduce about the closure properties of formally \'etale coalgebras is the following
    \begin{enumerate}
        \item Since the product of coalgebras is given by the underlying tensor product and base
        change is symmetric monoidal, finite products of formally \'etale coalgebras are formally 
        \'etale.
        \item Let $\kappa$ be a regular cardinal such that $\cCAlg(\cC_R)$ is $\kappa$-presentable.
              Then $\kappa$-filtered colimits of formally \'etale coalgebras are formally \'etale.
    \end{enumerate}
    One may contemplate the a priori weaker condition where we only ask that $L_A(\Sigma^nR)=0$ for
    all $n\geq 0$. This admits stronger closure properties, for example it is closed under
    limits since $R\oplus \Sigma^nR$ is dualizable as an $R$-module. Moreover, one 
    can combine Proposition~\ref{structure} with Proposition~\ref{etalchar} to see that this is
    equivalent to $L_A(M)$ vanishing for all dualizable $M$. With appropriate modification, the results
    of this section also apply to this notion of formally \'etale coalgebras. However, working
    with this notion makes lifting along iterated square zero extensions more annoying and we do not 
    know of an example of a coalgebra which satisfies this condition but not the stronger one of
    Definition~\ref{formalet}.
\end{remark}

\begin{definition}\label{derivations}
Let $R$ be a connective $\bb{E}_\infty$-ring, $A\in \rm{cCAlg}(\cl{C}_R)$ 
and $M \in \Mod_R^{\rm{cn}}$. For any other $B\in \cCAlg(\cl{C}_R)$ and a 
map $\varphi : B\to A$ we define the \textit{spectrum of derivations} from 
$B$ to $M$ as the mapping spectrum
  \[ \rm{der}_{\varphi}(B, M):= 
  \rm{map}_{\Sp(\rm{cCAlg}(\cl{C}_{R})_{/A})}(\Sigma^\infty_+B, L_A(M)).\]
  We also write 
  \[ \rm{Der}_{\varphi}(B, M) := \Omega^\infty \rm{der}_\varphi(B,M) \]
  for the underlying space.
\end{definition}

\begin{proposition}\label{maplifts}
  Let $R$ be a connective $\bb{E}_\infty$-ring, $M \in \Mod_R^{\rm{cn}}$
  and let $\cl{X}(\blank)=\rm{cCAlg}(\cl{C}_{\blank})$.
  Moreover, let $\varphi: B \to A$ a map in $\cCAlg(\cl{C}_{R})$
  i.e.~a point $\varphi \in \cl{X}^{\Delta^1}(R)$.
  We have natural equivalences
  \[ T^M_{\cl{X}^{\Delta^1}_\varphi} \simeq \rm{der}_\varphi (B, M)\]
  \[ T^M_{\cl{X}^{\Delta^0}_A} \simeq \rm{der}_{\id} (A, M[1]).\]
\end{proposition}
\begin{proof}
The first equivalence is clear from Construction~\ref{pi}. For the second, let $e:R \to R\oplus M[1]$
denote the 0-section. Observe that, since $\cl{X}^{\Delta^0}$ is cohesive we have natural equivalences
\begin{align*}
 (\cl{X}^{\Delta^0})_{A}^{R\oplus M} &\simeq \Omega (\cl{X}^{\Delta^0})_{A}^{R\oplus M[1]}\\
  &\simeq  \rm{fib}_{\id_{A}}(\rm{Map}_{\rm{cCAlg}(\cl{C}_{R\oplus M[1]})}(e^{\pt}A, e^{\pt}A )
  \to  \rm{Map}_{\rm{cCAlg}(\cl{C}_R)}(A, A))\\
  &\simeq \rm{Der}_{\id} (A, M[1]),
\end{align*}
as claimed.
\end{proof}

\begin{corollary}\label{defobject}
  Let $A\in \rm{cCAlg}(\cl{C}_{R})$ be formally \'etale and $R^{\eta} \to R$ a square zero extension 
  with fiber $M$. Then the space
  \[\rm{fib}_{A}\left(\rm{cCAlg}(\cl{C}_{R^{\eta}}) \to \rm{cCAlg}(\cl{C}_{R}) \right)\]
  is contractible, i.e.~$A$ admits an essentially unique lift to a coalgebra in $\cl{C}_{R^\eta}$
\end{corollary}
\begin{proof}
Indeed, since $\cl{X}=\cCAlg(\cC_{\blank})^{\Delta^0}$ is cohesive and $L_A\simeq 0$  Proposition~\ref{maplifts} implies that
\[ T^M_{\cl{X}^{\Delta^0}_A} \simeq \rm{der}_{\id} (A, M[1]) 
= \rm{Map}_{\Sp(\rm{cCAlg}(\cl{C}_{R})_{/A})}(\Sigma^\infty_+B, L_A(M[1])) \simeq 0.\]
Hence, the claim follows from Proposition~\ref{def}.
\end{proof}

For dualizable coalgebras, our notions of formally \'etale and derivations are compatible with
the usual notions from higher algebra.

 \begin{proposition}\label{cotangentder}
   Let $R$ be a connective $\bb{E}_{\infty}$-ring and assume that $B,A\in \rm{cCAlg}(\Mod^{\cn}_{R})$
   with $A$ dualizable. For every map $\varphi:B \to A$ with $R$-linear dual
   $\varphi^\vee: A^\vee \to B^\vee$ and every $M \in \Mod^{\cn}_R$ we have
  \[ \rm{Der}_{\varphi}(B, M) 
  \simeq \rm{Map}_{\rm{Mod}_{A^\vee}}(L_{A^\vee/R}, \varphi{\pt}\rm{map}_{R}(B, M)).\]
 \end{proposition}
\begin{proof}
Write $R\p = R \oplus M$ and denote by $e: R \to R\p$ the 0-section.
  By construction, the space $\rm{Der}_{\varphi}(B, M)$ is equivalent to the fiber
  \[ F_{\varphi}:=\rm{fib}_{\varphi}\left(\rm{Map}_{\rm{cCAlg}(\Mod_{R\p})}(e^\pt B,e^\pt A^\vee)
      \to \rm{Map}_{\rm{cCAlg}(\Mod_{R})}(B, A^\vee) \right)\]
  Since $A$ is dualizable, so is $e^\pt A$ with dual given by 
  $(e^\pt A)^{\vee}\simeq A^{\vee}\otimes_{R} R\p$.
  Thus, by Corollary~\ref{dualad}, applying $(\blank)^{\vee}$ yields an equivalence
  \begin{align*}
    \rm{Map}_{\rm{cCAlg}(\Mod_{R\p})}(e^\pt B, e^\pt A) &\simeq 
    \rm{Map}_{\rm{CAlg}(\Mod_{R\p})}(A^{\vee}\otimes_{R}R\p, \rm{map}_{R\p}(e^\pt B, R\p))\\
    &\simeq \rm{Map}_{\rm{CAlg}(\Mod_{R})}(A^{\vee}, \rm{map}_{R}(B, R\p))
  \end{align*}
  and similarly
  \begin{align*}
    \rm{Map}_{\rm{cCAlg}(\Mod_{R})}(B,A)\simeq \rm{Map}_{\rm{CAlg}(\Mod_{R})}(A^{\vee}, B^{\vee}).
  \end{align*}
  Thus, $F_{\varphi}$ is given by
  \begin{align*}
    F_{\varphi}&\simeq 
    \rm{fib}_{\varphi^{\vee}}\left(\rm{Map}_{\rm{CAlg}(\Mod_{R})}(A^{\vee}, \rm{map}_{R}(B, R\p))
    \to \rm{Map}_{\rm{CAlg}(\Mod_{R})}(A^{\vee}, B^{\vee})\right)\\
    &\simeq \rm{Map}_{(\rm{CAlg}(\Mod_{R}))_{/B^\vee}}(A^{\vee}, \rm{map}_{R}(B, R\p)),
  \end{align*}
  i.e.~the space of lifts in the diagram
\[\begin{tikzcd}
	& {\rm{map}_R(B,R\p)} \\
	{A^\vee} & {B^\vee}
	\arrow[from=2-1, to=2-2]
	\arrow[from=1-2, to=2-2]
	\arrow[dashed, from=2-1, to=1-2].
\end{tikzcd}\]
Since $R\p\to R$ is a split square zero extension with fiber $M$, the map
$\rm{map}_{R}(B, R\p) \to B^{\vee}$ is a square zero extension as well with fiber
$\rm{map}_{R}(B, M)$. Hence, we have that
\begin{align*}
F_{\varphi}\simeq \rm{Map}_{A^{\vee}}(L_{A^{\vee}_{R}/R}, \varphi^{\vee}_{\pt}\rm{map}_{R}(B,M))
\end{align*}
as claimed.
\end{proof}

This provides us with the most accessible examples of formally \'etale coalgebras.

\begin{corollary}\label{dualetal}
  Let $R$ be a connective $\bb{E}_{\infty}$-ring and $A\in \rm{cCAlg}(\Mod^{\cn}_{R})$ be dualizable
  such that the relative cotangent complex $L_{A^{\vee}/R}$ vanishes. Then $A$ is formally \'etale.
\end{corollary}

\begin{example}
    For any $X\in \cl{S}^\omega$ the $\F_p$-homology $\F_p[X]$ is compact and hence dualizable
    in $\Mod_{\F_p}$ with dual given by the cohomology $\F_p^X$. By~\cite[][Proposition 2.4.12.]{dag8}
    we have $L_{\F_p^X/\F_p} \simeq 0$. Hence, by Corollary~\ref{dualetal} we see that $\F_p[X]$
    is a formally \'etale coalgebra. Crucially, this reasoning does not apply to Eilenberg-MacLane
    spaces, for which we give a more direct proof in \ref{homocoalg}, which allows
    us to drop the finiteness assumption on the space $X$.
\end{example}

\begin{remark}
 It is unclear whether the converse of Corollary~\ref{dualetal} holds.
 From Proposition~\ref{cotangentder} we can only deduce that
\[ \rm{Map}_{A^{\vee}}(L_{A^{\vee}/R}, \varphi^{\vee}_{\pt}\rm{map}_{R}(B, M)) \simeq 0\]
for each coalgebra $B$, $R$-module $M$ and morphism of algebras $\varphi:A^{\vee}\to B^{\vee}$.
\end{remark}

It is clear in the split case that maps into formally \'etale coalgebras also admit unique lifts.
To make precise what happens in the non-split case, let us first describe the space of lifts
in the general case.

\begin{proposition}\label{fibpb}
  Let $X(\blank)=(\rm{cCAlg}(\cl{C}_{\blank})^{\Delta^{1}}$, $R$ an $\bb{E}_{\infty}$-ring
  and $(A\rar{\varphi} B) \in X(R)$. Then for every connective $R$-module $M$ the fiber
  $X^{R\oplus M}_{\varphi}$ can be computed as the pullback
  \[ X^{R\oplus M}_{\varphi} \simeq \rm{Der}_{\rm{id}}(B, M[1]) \times_{\rm{Der}_{\varphi}(B, M[1])}
    \rm{Der}_{\rm{id}}(A, M[1]).\]
\end{proposition}
\begin{proof}
Since $X$ is cohesive, the space $X_{\varphi}^{R\oplus M}$ fits into a pullback diagram
\[\begin{tikzcd}
	{X_{\varphi}^{R\oplus M}} & \pt \\
	\pt & {X_{\varphi}^{R\oplus M[1]}}
	\arrow[from=1-1, to=1-2]
	\arrow[from=1-1, to=2-1]
	\arrow[from=2-1, to=2-2]
	\arrow[from=1-2, to=2-2]
\end{tikzcd}\]
where the maps $\pt \to X_{\varphi}^{R\oplus M[1]}$ are given by the base changed
morphisms $\varphi \otimes_{R} (R \oplus M[1])$. \\
If $\cl{C}$ is any category and $(\psi:A\to B)\in \cl{C}^{\Delta^{1}}$, then $\Omega_{\psi}\cl{C}^{\Delta^{1}}$ is given
by the pullback
\[\begin{tikzcd}
	{\Omega_\psi\cl{C}^{\Delta^1}} & {\rm{Aut}_{\cl{C}}(A)} \\
	{\rm{Aut}_{\cl{C}}(B)} & {\rm{Map}_{\cl{C}}(A,B)}
	\arrow["{\blank \circ \psi}"', from=2-1, to=2-2]
	\arrow["{\psi \circ \blank}", from=1-2, to=2-2]
	\arrow[from=1-1, to=2-1]
	\arrow[from=1-1, to=1-2]
\end{tikzcd}.\]
Thus, we can compute the loop space $\Omega X_{\varphi}^{R\oplus M[1]}$ as the pullback
\[\begin{tikzcd}
	{\Omega X_{\varphi}^{R\oplus M[1]}} & {\rm{Der}_{\id}(A, M[1])} \\
	{\rm{Der}_{\id}(B, M[1])} & {\rm{Der}_{\id}(B, M[1])}
	\arrow[from=2-1, to=2-2]
	\arrow[from=1-2, to=2-2]
	\arrow[from=1-1, to=2-1]
	\arrow[from=1-1, to=1-2]
\end{tikzcd}\]
as claimed.
\end{proof}

In particular, if $A$ is formally \'etale this means that the space of lifts of $\varphi$
is equivalent to the space of lifts of $B$ to an $R\oplus M$-coalgebra so we get the following.

\begin{corollary}\label{defmaps}
  Let $R$ be a connective $\bb{E}_\infty$-ring, $A,B\in \rm{cCAlg}(\cl{C}_{R})$ with $A$ formally \'etale
  and let $q:R^{\eta} \to R$ be a square zero extension with fiber $M$. Suppose we are given lifts
  $A\p$ and $B\p$ of $A$ and $B$ respectively to
  $\rm{cCAlg}(\cl{C}_{R^{\eta}})$. Then the natural map
  \[q^\pt:\rm{Map}_{\rm{cCAlg}_{R^{\eta}}}(B\p, A\p) \to \rm{Map}_{\rm{cCAlg}_{R}}(B,A)\]
  is a homotopy equivalence.
\end{corollary}
\begin{proof}
Let $X(\blank)= \rm{cCAlg}(\cl{C}_{\blank})^{\Delta^1}$. Proposition~\ref{fibpb} implies that,
since $A$ is formally \'etale, we have $T_{X_\varphi}^M \simeq 0$ for any $\varphi: B\to A$. Thus,
each fiber of the map $q^\pt$ is contractible and the claim follows.
\end{proof}

We can also characterize formally \'etale coalgebras in this way, as the following proposition shows.

\begin{proposition}\label{etalchar}
    Let $R$ be a connective $\bb{E}_{\infty}$-ring and $A\in \rm{cCAlg}(\cC_R)$. We write
    $\cl{X}(R) = \rm{cCAlg}(\cC_R)$. Then $A$ is formally \'etale if and only if for every
    $B \in \rm{cCAlg}(\cC_R), M \in \rm{Mod}_R^{cn}$ and every morphism
    $\varphi:B\to A$, the map
    \[\rm{ev}_0:T_{(\cl{X}^{\Delta^1})_\varphi}^{M} \to T_{(\cl{X}^{\Delta^0})_{B}}^{M}\]
    induced by the evaluation at the domain is an equivalence.
\end{proposition}
\begin{proof}
    Write
  \[ F_{B\p}^M:=\rm{fib}_{B\p}\left((\cl{X}^{\Delta^{1}})_{\varphi}^{R\oplus M} 
  \to (\cl{X}^{\Delta^{0}})_{B}^{R\oplus M}\right)\]
  for the fiber over some point $B\p\in (\cl{X}^{\Delta^{0}})_{B}^{R\oplus M}$. 
  Then by definition of the tangent complex, the map $T_{(\cl{X}^{\Delta^1})_\varphi}^{M} \to T_{(\cl{X}^{\Delta^0})_{B}}^{M}$ is an equivalence
  if and only if we have $F_{B\p}^{M} \simeq 0$ for all $M\in \rm{Mod}_{R}^{\rm{cn}}, B\p \in (\cl{X}^{\Delta^{0}})_{B}^{R\oplus M}$.
  The two pasted pullback squares
\[\begin{tikzcd}
	F_{B\p}^{M} & {(\cl{X}^{\Delta^1})_{\varphi}^{R \oplus M}} & {\rm{Der}_{\id}(A, C_A(M[1]))} \\
	\ast & {\rm{Der}_{\id}(B,C_B(M[1]))} & {\rm{Der}_{\varphi}(B, C_A(M[1]))}
	\arrow[from=1-1, to=2-1]
	\arrow[from=2-1, to=2-2]
	\arrow[from=1-2, to=2-2]
	\arrow[from=1-1, to=1-2]
	\arrow[from=1-2, to=1-3]
	\arrow[""{name=0, anchor=center, inner sep=0}, from=2-2, to=2-3]
	\arrow[from=1-3, to=2-3]
	\arrow["\lrcorner"{anchor=center, pos=0.125}, draw=none, from=1-1, to=2-2]
	\arrow["\lrcorner"{anchor=center, pos=0.125}, draw=none, from=1-2, to=0]
\end{tikzcd}\]
yield a fiber sequence
\[F_{M} \to \rm{Der}_{\id}(A, C_{A}(M[1])) \rar{\blank \circ \varphi }  \rm{Der}_{\varphi}(B, C_{A}(M[1])).\]
Hence, the ``only if'' direction holds. Moreover, we see that if $F_{B\p}^{M} \simeq 0$ for every
$M\in \rm{Mod}_{R}^{\rm{cn}}, B\p \in (\cl{X}^{\Delta^{0}})_{B}^{R\oplus M}$ and any morphism $B\rar{\varphi} A$, we obtain
a zigzag of equivalences
\[ \rm{Der}_{\varphi}(B, C_{A}(M[1])) \xleftarrow{\sim} \rm{Der}_{\id}(A, C_{A}(M[1]))
  \rar{\sim} \rm{Der}_{0}(0, C_{A}(M[1])) \simeq \pt,\]
where $0 \in \rm{cCAlg}_{k}$ denotes the initial coalgebra. Thus, we have that
\[\rm{Der}_{\varphi}(B, C_{A}(M))\simeq \Omega \rm{Der}_{\varphi}(B, C_{A}(M[1]))\simeq \pt\]
as claimed.
\end{proof}

\begin{corollary}\label{etallift}
  Let $R$ be a connective ring spectrum and $A\in \rm{cCAlg}(\cl{C}_{\blank})$ be formally \'etale.
  For a square zero extension $R^{\eta} \to R$ with fiber $M$ denote by $A^{\eta}$ the essentially
  unique lift of $A$ to $\cCAlg(\cC_{R^\eta})$. Then $A^{\eta}$ is also formally \'etale.
\end{corollary}
\begin{proof}
  Let $B \to A^{\eta}$ be any map of $R^{\eta}$ coalgebras and write 
  $\cl{X}(\blank)= \rm{cCAlg}^{\rm{cn}}_{\blank}$.
  Then for any $N\in \rm{Mod}_{R^{\eta}}^{\rm{cn}}$ we need to show that the induced map
  \[ T^{N}_{\cl{X}^{\Delta^{1}}_{\varphi}} \rar{\sim} T^{N}_{\cl{X}^{\Delta^{0}}_{B}}\]
  is an equivalence. Arguing as in the proof of Proposition~\ref{cofib}, we see that $N$
  sits in a cofiber sequence
  \[ M\otimes_{R}(R \otimes_{R^{\eta}} N) \to N \to R \otimes_{R^{\eta}}N, \]
  where the $R^{\eta}$-action on the outer terms factors through $R$.
  Thus, writing $B\p \simeq B\otimes_{R^{\eta}} R$ and $\varphi\p = \varphi_{R^{\eta}}: B\p \to A$
  and using that the tangent complex functors are excisive, we obtain a commutative diagram
\[\begin{tikzcd}
	{T_{\cl{X}^{\Delta^1}_{\varphi\p}}^{M\otimes_{R}(R\otimes_{R^\eta} N)}} & {T^{M\otimes_{R}(R\otimes_{R^\eta} N)}_{\cl{X}^{\Delta^1}_\varphi}} & {T^N_{\cl{X}^{\Delta^1}_\varphi}} & {T^{R\otimes_{R^\eta} N}_{\cl{X}^{\Delta^1}_\varphi}} & {T_{\cl{X}^{\Delta^1}_{\varphi\p}}^{R\otimes_{R^\eta} N}} \\
	{T_{\cl{X}^{\Delta^0}_{B\p}}^{M\otimes_{R}(R\otimes_{R^\eta} N)}} & {T_{\cl{X}^{\Delta^0}_B}^{M\otimes_{R}(R\otimes_{R^\eta} N)}} & {T^N_{\cl{X}^{\Delta^0}_B}} & {T^{R\otimes_{R^\eta} N}_{\cl{X}^{\Delta^0}_B}} & {T_{\cl{X}^{\Delta^0}_{B\p}}^{R\otimes_{R^\eta} N}}
	\arrow[from=1-2, to=1-3]
	\arrow[from=1-3, to=1-4]
	\arrow[from=1-2, to=2-2]
	\arrow[from=2-2, to=2-3]
	\arrow[from=2-3, to=2-4]
	\arrow[from=1-4, to=2-4]
	\arrow[from=1-3, to=2-3]
	\arrow["\sim", from=1-4, to=1-5]
	\arrow["\sim", from=2-4, to=2-5]
	\arrow["\sim", from=1-5, to=2-5]
	\arrow["\sim"', from=1-2, to=1-1]
	\arrow["\sim", from=1-1, to=2-1]
	\arrow["\sim"', from=2-2, to=2-1]
\end{tikzcd}\]
where the inner two horizontal maps in each row form a cofiber sequence. 
The outer horizontal maps are the base change equivalences from Proposition~\ref{bc} and the outer
vertical maps are equivalences since by assumption $A= A^{\eta}\otimes_{R^{\eta}}R$
is formally \'etale. Thus, the middle map 
$T^{N}_{{\cl{X}^{\Delta^{1}}_{\varphi}}}\to T^{N}_{\cl{X}^{\Delta^{0}}_{B}}$
is an equivalence as well, so by Proposition~\ref{etalchar} the $R^{\eta}$-coalgebra $A^{\eta}$
is formally \'etale.
\end{proof}

We can neatly organize the results of this section into the following proposition.

\begin{proposition}\label{weillift}
    Let $R$ be a connective $\bb{E}_\infty$-ring and $q: R^\eta \to R$ be a square zero extension. 
    Weil-restriction along $q$ induces a fully faithful functor
    \[ q_\pt : \cCAlg(\cl{C}_R)^{\rm{f\acute{e}t}} \to \cCAlg(\cl{C}_{R^\eta})^{\rm{f\acute{e}t}}\]
    Moreover, for any $A\in \cCAlg(\cl{C}_R)^{\rm{f\acute{e}t}}$ the coalgebra $q_\pt A$ is, up to contractible
    choice, the unique lift of $A$ to $\cCAlg(\cl{C}_{R^\eta})$.
\end{proposition}
\begin{proof}
By Corollary~\ref{defobject}, there exists a unique lift $A\p \in \cCAlg(\cl{C}_{R^\eta})$ which is formally 
\'etale by Corollary~\ref{etallift}. Moreover, by Corollary~\ref{defmaps}, we have a natural equivalence 
\[ \Map_{\cCAlg(\cl{C}_{R^\eta})}(B, A\p) \rar{\sim} \Map_{\cCAlg(\cl{C}_R)}(p^\pt A, A)\]
and so $A\p = q_\pt A$. Since $A\p$ was a lift of $A$, the counit $q^\pt q_\pt A \rar{\sim} A$ is an equivalence
which implies that the restriction of $q_\pt$ is fully faithful as claimed.
\end{proof}

We can also iterate lifting along square zero extensions into a limit. This will be crucial in
the next section for going from $\F_p$ to the $p$-completed sphere $\S_p^\wedge$.

\begin{proposition}\label{limlift}
    Suppose we are given a diagram 
    \[ \dots \to R_2 \rar{q^{21}} R_1 \rar{q^{10}} R_0\]
    in $\CAlg^{\cn}$ with limit $R= \lim_n R$ such that each map $R_{n+1} \to R_{n}$ is a square
    zero extension and the functor $\cl{C}_R \to \lim_n \cl{C}_{R_n}$ is an equivalence. Write
    $q:R\to R_0$ for the natural map. Weil restriction along $q$ restricts to a fully faithful functor
    \[ q_\pt: \cCAlg(\cl{C}_{R_0})^{\fet} \to \cCAlg(\cl{C}_{R}),\]
    such that for any $A\in \cCAlg(\cl{C}_{R_0})^{\fet}$ the coalgebra $q_\pt A$ is, up to contractible choice,
    the unique lift of $A$ to an object of $\cCAlg(\cl{C}_{R})$.
\end{proposition}
\begin{proof}
First observe that by Lemma~\ref{limits} the map
\[ \cCAlg(\cl{C}_R) \to \lim \cCAlg(\cl{C}_{R_n})\]
is also an equivalence. Denote by $q^n$ the map $R_n \to R$. 
Then $q_\pt: \cCAlg(\cl{C}_{R_0})\to \cCAlg(\cl{C}_{R})$ is given
   by the limit of the functors
   \[ q^n_\pt:\cCAlg(\cl{C}_{R_0}) \to \cCAlg(\cl{C}_{R_n}). \]
   We have that $q^n_\pt \simeq q^{10}_\pt \circ q^{21}_\pt \circ \dots \circ q^{n(n-1)}_\pt$ and for each 
   $k$ the restriction
   \[ q^{k(k-1)}_\pt :\cCAlg(\cl{C}_{R_{k-1}})^{\fet} \to \cCAlg(\cl{C}_{R_k})^{\fet}\]
   is fully faithful by Proposition~\ref{weillift}. 
   Thus, $q^n_\pt$ is fully faithful when restricted to the formally \'etale coalgebras. 
   Moreover, since mapping spaces in a limit of categories are given by the limit of mapping spaces, a 
   limit of fully faithful functors is fully faithful and so $q_\pt$ is fully faithful when restricted
   to formally \'etale coalgebras. The fact that $q_\pt A$ is the unique lift of $A$ is immediate
   by Proposition~\ref{weillift} and the fact that fibers commute with limits.
\end{proof}
\section{Homology and spherical Weil-restriction}
\subsection{Lifting to the spherical Witt vectors}
Fix a prime $p$. We now collect the fruits of our labor and apply the machinery of formally 
\'etale coalgebras to the moduli problem
\[ \Mod_\blank^\wedge: \CAlg^{\cn} \to \CAlg(\rm{Pr^L}),\]
where $\Mod_R^\wedge$ denotes the category of $p$-complete $R$-modules. This is a legal move by 
Corollary~\ref{pnil}. Throughout this section we fix a perfect $\F_p$-algebra $k$ and 
denote the spherical Witt vectors by $\bb{W}(k)$ and the $p$-typical Witt Vectors as
$\bb{W}_0(k)= \pi_0\bb{W}(k)$.
Before we can state our main theorems, we require one more technical result about $p$-adic lifts
of formally \'etale coalgebras.

\begin{proposition}\label{etalliftzp}
 Let $A \in \cCAlg(\Mod_{\W_0(k)}^{\cn,\wedge})$ 
  Then $A$ is formally \'etale if $A_{1}:=A \otimes k \in \rm{cCAlg}(\Mod_k^{\cn})$
  is formally \'etale.
\end{proposition}
\begin{proof}
   Write $\cl{X}= (\rm{cCAlg}_{\blank}^{\rm{cn}})^{\wedge}_{p}$
   and let $\varphi: B \to A$ be a point in $\cl{X}(\W_0(k))^{\Delta^{1}}$.
   Denote for each $n\in \Z_{\geq 1}$ by $\varphi_{n}: B_{n}\to A_{n}$ the base change of 
   $\varphi$ to $\W_0(k)/p^{n}$. By Proposition~\ref{etalchar} it suffices to show that for
   every $\W_0(k)_{p}$-module $M$ the natural map
   \[T_{(\cl{X}^{\Delta^1})_\varphi}^{M} \to T_{(\cl{X}^{\Delta^0})_{B}}^{M}\]
  is an equivalence. This map is induced by the maps of spaces
  \[ (\cl{X}^{\Delta^{1}})_{\varphi}^{\W_0(k)\oplus M [k]} 
  \to (\cl{X}^{\Delta^{0}})_{B}^{\W_0(R) \oplus M[k]}\]
  and hence it suffices to show that these are equivalences. Since fibers commute with limits
  we have for each $i \in \Z_{>0}$, $\eta \in \cl{X}^{\Delta^{i}}(\W_0(R)$ with base change
  $\eta_{n}$ to $\W_0(R)/p^{n}$ and each $\W_0(R)$-module $N$ an equivalence
  \[(\cl{X}^{\Delta^{i}})_{\eta}^{\W_0(R)\oplus N} 
  \simeq \lim_{n} (\cl{X}^{\Delta^{i}})_{\eta_{n}}^{\W_0(R)/p^{n}\oplus N/p^{n}}.\]
  As each $A_{n}$ is a lift of $A_{1}$ to $\W_0(R)/p^{n}$ and hence by Corollary~\ref{etallift}
  formally \'etale, the map
  \[ (\cl{X}^{\Delta^{1}})_{\varphi_{n}}^{\W_0(R)/p^{n} \oplus M/p^{n}} 
  \to (\cl{X}^{\Delta^{1}})_{B_{n}}^{\W_0(R)/p^{n} \oplus M/p^{n}}\]
  is an equivalence for all $n$ and the claim follows.
\end{proof}

We now summarize our results in our first main theorem, namely that the right adjoint to base change
defines a $p$-typical Witt vector style functor for formally \'etale connective coalgebras over $\F_p$.

\begin{theorem}\label{wittzp}
    Let $q: \W_0(k) \to k$ be the augmentation. Connective Weil-restriction along $q$ induces a
    fully faithful functor
    \[\cl{W}_0: \rm{cCAlg}(\Mod_k^{\cn})^{\fet} \to \rm{cCAlg}(\Mod_{\W_0(k)}^{\cn,\wedge})^{\fet}.\]
    Moreover, for any $A \in \rm{cCAlg}(\Mod_{k}^{\rm{cn}})^{\fet}$ the coalgebra $\cl{W}_0(A)$
    is, up to contractible choice, the unique lift of $A$ to an object
    of $\rm{cCAlg}(\Mod_{\W_0(k)}^{\wedge})$.
\end{theorem}
\begin{proof}
  By Proposition~\ref{pcomp} we may apply Proposition~\ref{limlift} to the functor 
  \[\CAlg^{\cn}\to \CAlg(\rm{Pr^L}) \quad R \mapsto \Mod_R^{\cn,\wedge}\]
  and the $p$-completion tower 
 \[ \dots \to \W_0(k)/p^2 \to \W_0(k)/p = k\]
  to get fully faithfulness and the uniqueness of the lifts.
  Then Proposition~\ref{etalliftzp} tells us that the lift is again formally \'etale and we are done.
\end{proof}

\begin{proposition}
    Let $A \in (\rm{cCAlg}_{W(k)}^{\rm{cn}})^\wedge_p$, then $A$ is formally \'etale if
    $A\otimes W_0(k)$ is.
\end{proposition}
\begin{proof}
    Argue exactly as in Proposition~\ref{etalliftzp}.
\end{proof}

With this in hand, we obtain a spherical lift of Theorem~\ref{wittzp}.

\begin{theorem}\label{wittsp}
    Let $q: \W(k) \to k$ be the augmentation. Connective Weil-restriction along $q$ induces a fully
    faithful functor
    \[\cl{W}: \rm{cCAlg}(\Mod_k^{\cn})^{\fet} \to \rm{cCAlg}(\Mod_{\W(k)}^{\cn,\wedge})^{\fet}.\]
    Moreover, for any $A \in \rm{cCAlg}(\Mod_k^{\rm{cn}})^{\fet}$ 
    the coalgebra $\cl{W}(A)$ is, up to contractible choice, the unique lift of $A$ to an object 
    of $\rm{cCAlg}(\Mod_{\W(k)}^{\wedge})$.
\end{theorem}
\begin{proof}
By Theorem~\ref{wittzp} it suffices to show that Weil restriction along
$\W(k)\to \W_0(k)=\pi_0\W(k)$ is fully faithful on formally \'etale coalgebras. By Lemma~\ref{pnil} we can again apply Proposition~\ref{limlift} to the
functor 
  \[\CAlg^{\cn}\to \CAlg(\rm{Pr^L}) \quad R \mapsto \Mod_R^{\cn,\wedge}\]
  and the Postnikov tower
  \[ \dots \to \tau_{\leq 2} \W(k) \to \tau_{\leq 1}\W(k) \to \tau_{\leq 0}\W(k) = \W_0(k)\]
  and win.
\end{proof}

\subsection{Homology coalgebras}\label{homocoalg}

As observed in Example~\ref{homology}, for every space $X$ and every $\bb{E}_{\infty}$-ring $R$, the
$R$-homology $R[X]$ carries a natural $R$-coalgebra structure induced from the diagonal
map $X\to X\times X$. The main result of this section is Theorem~\ref{sepet}, namely that if
$R$ is a connective $\F_p$-algebra and $X$ is a connected space, the coalgebra $R[X]$ is
formally \'etale. Our proof relies on a recent computation of Bachmann\textendash Burklund in \cite{bb}.
Moreover, combining Theorem~\ref{sepet} and Theorem~\ref{wittsp} with the main result of
loc.cit.~yields a further embedding of $p$-complete nilpotent spaces into coalgebras in 
$p$-complete $\bb{W}(\overline{\F}_p)$-modules.
In the following, let $k$ denote a field of characteristic $p$ unless stated otherwise.\\

We begin by quickly recalling some facts about pro-objects, see loc.cit.~Section 5 for a more 
in-depth discussion. Let $\cl{C}$ be a category which admits finite limits. Recall that,
the category $\rm{Pro}(\cl{C})$ is defined as the full subcategory of 
$\rm{Fun}(\cl{C}, \cl{S})\op$ spanned by those functors which preserve finite limits.
We think of objects of $\Pro(\cC)$ as formal cofiltered limits of objects of $\cC$. 
If $\cC$ admits all small limits, the restricted Yoneda embedding gives an adjunction
\[\begin{tikzcd}
	{c: \cC} & {\Pro(\cC): M}
	\arrow[""{name=0, anchor=center, inner sep=0}, shift left=2, from=1-1, to=1-2]
	\arrow[""{name=1, anchor=center, inner sep=0}, shift left=2, from=1-2, to=1-1]
	\arrow["\dashv"{anchor=center, rotate=-90}, draw=none, from=0, to=1].
\end{tikzcd}\]
Here we think of $c$ (aka restricted Yoneda) as taking $X\in \cC$ to the constant diagram on $X$.
The functor $M$ is called materialization and takes a formal limit $\flim_\lambda X_\lambda$ to 
the actual limit computed in $\cC$. An object $ \flim X_\lambda \in \rm{Pro}(\Mod_R)$ is called \textit{pro-truncated} if each $X_\lambda$ is
bounded above. The inclusion of pro-truncated objects admits a left adjoint denoted $\tau_{<\infty}$,
which can explicitly described by the formula
\[ \tau_{<\infty} \lim_{\lambda} X_\lambda = \lim_{n, \lambda} \tau_{\leq n} X_{\lambda}. \]
An object $E \in \rm{Pro}(\Mod_R)$ is called \textit{pro-constant up to pro-truncation} if the natural
map
\[ \tau_{<\infty} c M(E) \to \tau_{<\infty} E\]
is an equivalence in $\rm{Pro}(\Mod_R)$. If $\cC$ is symmetric monoidal, then $\rm{Pro}(\cC)$
inherits a natural symmetric monoidal structure via the pointwise tensor product such that the
inclusion $\cC \to \rm{Pro}(\cC)$ is monoidal. We denote the partially defined right adjoint
to $c: \cCAlg(\cC) \to \cCAlg(\Pro(\cC))$ by $M^{cA}$ and refer to it as
\textit{coalgebraic materialization}. Moreover, we write 
\[ \widehat{C}: \Pro(\cC) \to \cCAlg(\Pro(\cC))\]
for the pro-version of the cofree coalgebra functor. Observe that, for any $X\in \cC$ 
the coalgebraic materialization $M^{cA}\widehat{C}(cX)\in \cCAlg(\cC)$ exists and is given by
$C(X)$. The first key insight is that, although $\Pro(\cC)$ is usually not presentable even if $\cC$ is, 
the category $\cCAlg(\Pro(\cC))$ is better behaved with respect to limits.

\begin{lemma}\label{probc1}
    Let $f:R\to S$ be a map of $\bb{E}_\infty$-rings. The base change 
    \[ f^\pt:\rm{cCAlg}(\rm{Pro}(\Mod_R)) \to \cCAlg(\rm{Pro}(\Mod_S)) \]
    preserves all limits, cofree coalgebras and constant objects.
\end{lemma}
\begin{proof}
    The base change $\Mod_R \to \Mod_S$ commutes with finite limits and is accessible. Thus,
    the induced functor $f^\pt:\rm{Pro}(\Mod_R) \to \rm{Pro}(\Mod_S)$ commutes with all limits. 
    In particular for any $ M\in \rm{Pro}(\Mod_R)$ the cofree coalgebra on $M$ is computed 
    by the formula
    \[ \hat{C}_R(M) \simeq \prod_n (M^{\otimes n})^{h\Sigma_n} \in \cCAlg(\Pro(\Mod_R)),\]
    and thus, since $f^\pt$ commutes with tensor products and limits we get
    \[ f^\pt \hat{C}_R(M) = \hat{C}_S(f^\pt M)\]
    as claimed. Finally, the fact that $f^\pt$ preserves constant objects is clear,
    since it is induced by $\rm{Pro}(\blank)$.
\end{proof}

\begin{lemma}\label{probc2}
   Let $f:R\to S$ be a map of connective $\bb{E}_\infty$-rings and let $E$ be a connective
   $R$-module. Then the base change $f^\pt(\tau_{<\infty}cE) \in \rm{Pro}(\Mod_S)$ is
   pro-constant up to pro-truncation.
\end{lemma}
\begin{proof}
We need to show that the natural map
\[ \tau_{<\infty} c M(f^\pt(\tau_{<\infty}cE )) \to \tau_{<\infty}(f^\pt\tau_{<\infty}cE)\]
is an equivalence. Unraveling the definitions, we see that
 \[\tau_{<\infty} c M(f^\pt(\tau_{<\infty}cE ))
 = \tau_{<\infty}c(\lim_n ((\tau_{\leq n} M) \otimes_R S))
 \simeq \tau_{<\infty} c(E\otimes_R S) \simeq \flim_i c(\tau_{\leq i}(E\otimes_R S)),\]
 where we have used that, since everything is connective, the natural map
 \[  E \otimes_R S \rar{\sim} \lim_n\tau_{\leq n} E \otimes_R S  \]
 is an equivalence. Conversely, we have that 
 \[\tau_{<\infty}(f^\pt\tau_{<\infty}cE)
 \simeq \flim_{k,n} c(\tau_{\leq k}((\tau_{\leq n} E)\otimes_R S)).\]
 Again using the connectivity of everything in sight, we can find a cofinal system 
 of pairs $(n,k)$ with $n>>k$ such that
 \[ \tau_{\leq k}((\tau_{\leq n}E) \otimes_R S) \simeq \tau_{\leq k} (E \otimes_R S),\]
and so the claim follows. 
\end{proof}

The second key insight of is the \textit{Arint-Schreier map} on the cofree coalgebra. We quickly
recall its construction and generalize it slightly.

\begin{construction}
The functor
\[ (\blank)^\vee=\rm{map}_k(\blank , k): \cCAlg(\Mod_k) \to \Mod_k\]
is represented by the spectrum object $\{C_k(\Sigma^n k)\}$. Hence, on any $A\in \cCAlg(\Mod_k)$
the power operation 
\[ Q_0 - 1: A^\vee \to A^\vee\]
induces a map
\[ F-1: C_k(\Sigma^n k) \to C_k(\Sigma^nk)\]
called the \textit{Artin\textendash Schreier} map. In \cite[][Construction 7.4]{bb},
for every pointed set $X$ and any $n\geq 0$ a pro Artin\textendash Schreier map
\[ \widehat{F}-1 :\widehat{C}( c \Sigma^n k\{X\}) \to \widehat{C}(c\Sigma^n k\{X\})\]
in $\cCAlg(\rm{Pro}(\Mod_k))$ is defined, which is natural in inert maps of pointed sets and 
recovers the ordinary Artin\textendash Schreier map upon coalgebraic materialization. 
Now let $S$ be a $k$-algebra.
Since pro-base change commutes with cofree coalgebras, we can base change
$\widehat{F}-1$ along any map of $\E_\infty$-rings $k \to R$ to obtain a map
\[ \widehat{F}_S-1: \widehat{C}_S(c \Sigma^n S\{X\}) \to \widehat{C}_S(c\Sigma^n S\{X\}).\]
We define the ($S$-linear) Artin\textendash Schreier map 
\[ F_S-1: C_S(\Sigma^n S\{X\}) \to C_S(\Sigma^n S\{X\})\]
as the materialization of $\widehat{F}_S-1$. 
\end{construction}

\begin{remark}\label{frobres}
Let $f:R\to S$ be a map of $k$-algebras. We can Weil restrict along $f$ to obtain a modified
Artin\textendash Schreier map 
\[ f_!(F_S-1): C_R(\Sigma^n S\{X\}) \to C_R(\Sigma^n S\{X\}).\]
Tracing through the adjuncions, we see that for $X=\pt$, this represents the operation $Q_0-1$
acting on $\rm{map}_k(A,S) \in \CAlg(\Mod_k)$.
\end{remark}

\begin{proposition}\label{bbpb}
    Let $k$ be a field of characteristic $p$ and $R \in \rm{CAlg}(\rm{Mod}_k)^\rm{cn}$.
    For every $X\in \rm{Set}_\pt$ and $n\geq 0$, the Artin\textendash Schreier map $F_R$ induces a pullback
    diagram in $\rm{cCAlg}(\rm{Mod}_R)$ of the form
   \[\begin{tikzcd}
	{R[\Omega^\infty\Sigma^n \F_p\{X\}]} & {C_R(\Sigma^n R\{X\})} \\
	R & {C_R(\Sigma^n R\{X\})}
	\arrow[from=1-1, to=1-2]
	\arrow["{F_{R}-1}", from=1-2, to=2-2]
	\arrow[from=1-1, to=2-1]
	\arrow[', from=2-1, to=2-2].
\end{tikzcd}\] 
\end{proposition}
\begin{proof}
The proof of~\cite[][Proposition 7.10]{bb} implies that we have a pullback diagram in 
$\cCAlg(\rm{Pro}(\Mod_k))$ of the form
\[\begin{tikzcd}
	{\tau_{<\infty}ck[\Omega^\infty\Sigma^n\F_p\{X\}]} & {\widehat{C}_k(c\Sigma^n k\{X\})} \\
	ck & {\widehat{C}_k(c\Sigma^n k\{X\})}
	\arrow[from=1-1, to=1-2]
	\arrow[from=1-1, to=2-1]
	\arrow[from=2-1, to=2-2]
	\arrow["{\widehat{F}-1}",from=1-2, to=2-2].
\end{tikzcd}\]
By Lemma~\ref{probc1} base change along the unit $k \to R$
yields a pullback square in $\cCAlg(\rm{Pro}(\Mod_R))$ of the form
\begin{equation}\label{pullback}
\begin{tikzcd}
	{\tau_{<\infty}ck[\Omega^\infty\Sigma^n\F_p\{X\}] \otimes_k R}
    & {\widehat{C}_R(c\Sigma^n R \{X\})} \\
	cR & {\widehat{C}_R(c\Sigma^n R\{X\})}
	\arrow[from=1-1, to=1-2]
	\arrow[from=1-1, to=2-1]
	\arrow[from=2-1, to=2-2]
	\arrow["{\widehat{F}_R-1}",from=1-2, to=2-2].
\end{tikzcd}
\end{equation}
Since $k$ and $R$ are connective, Lemma~\ref{probc2} implies that 
\[\tau_{<\infty} ck[\Omega^\infty\Sigma^n\F_p\{X\}] \otimes_k R \]
is pro-constant up to pro-truncation.
Thus, by~\cite[][Lemma 5.10]{bb} the coalgebraic materialization agrees with
the underlying materialization and we see that
\[ M^{\rm{cA}}(\tau_{<\infty} (ck[\Omega^\infty \Sigma^n \F_p\{X\}] \otimes_k R)
\simeq  \lim_n ( (\tau_{\leq n}k[\Omega^\infty \Sigma^n\F_p\{X\}]) \otimes_k R)
\simeq R[\Omega^\infty \Sigma^n \F_p\{X\}],\]
where the limit is computed underlying by~\cite[][Lemma 5.11]{bb} and we again
use that all terms are connective. Since $M^{\rm{cA}}$
is a right adjoint it preserves limits and it is easy to verify from the universal property
that $M^{\rm{cA}}(\widehat{C}_R(c\Sigma^n R\{X\})$ exists and is given by $ C_R(\Sigma^n R\{X\})$.
Thus, applying $M^{\rm{cA}}$ to the pullback~\eqref{pullback} we get our claim.
\end{proof}

\begin{remark}\label{Q0}
Let $R \in \CAlg(\Mod_k), M \in \Mod_R $ and $p:R\oplus M \to R$ be the split square
zero extension of $R$ with fiber $M$. Denote by $e: R \to R \oplus M$ the 0-section.
As observed in Remark~\ref{cohsq0}, the multiplication map
\[
(M^{\otimes p})_{h\Sigma_p} \to M
\]
is nullhomotopic. Hence, the power operation $Q_0:R\oplus M \to R\oplus M$ is nullhomotopic
when restricted to $M$ and hence factors as
\[ R\oplus M \rar{p} R \rar{Q_0} R \rar{e} R\oplus M.\]
This enables us to understand the Weil-restricted Artin\textendash Schreier map of Remark~\ref{frobres}.
\end{remark}

\begin{proposition}\label{split}
Let $R\in \rm{CAlg}(\Mod_k^{\rm{cn}}), M \in \Mod_R^{\rm{cn}}$, $X\in \rm{Set}_\pt$
and denote by $e:R \to R \oplus M$ the 0-section of the split square zero extension.
The (non-connective) Weil-restriction of the map
$F_{R\oplus M}: C_{R\oplus M}(e^\pt R\{X\})\to C_{R\oplus M}(e^\pt R\{X\})$ 
along $e$ yields a map 
    \[ e_!(F_{R\oplus M}-1): C_R(e_\pt e^\pt R\{X\})\rar{}
    C_R(e_\pt e^\pt R\{X\}).\]
which naturally splits in $\cCAlg(\Mod_R)$ as
    \[ (F_R-1) \otimes_R \id: C_R( R\{X\} ) \otimes_R C_R(M\{X\}) 
    \rar{} C_R(R\{X\}) \otimes_R C_R( M\{X\}).\]
\end{proposition}
\begin{proof}
First note that, since $e_\pt e^\pt R\{X\} \simeq R\{X\} \times M \{X\}$ and the cofree 
coalgebra preserves limits, we obtain
\[ C_R(e_\pt e^\pt R\{X\}) \simeq C_R( R\{X\} ) \otimes_R C_R(M\{X\}).\]
To show that $e_!(F_R -1)$ is given by the product $(F_R-1) \otimes \id $, it suffices to check 
this after applying $\Map_{\cCAlg(\Mod_R)}(A, \blank)$ for some arbitrary $A \in \cCAlg(\Mod_R)$.
Since the functor $\rm{Map}_{\Mod_R}(A, \blank)$ commutes with limits, we see that for each
$n\geq 0$ we have 
\[ \Map_{\Mod_R}(A, (R\oplus M)\{S^n\}) \simeq
\Omega^{\infty-n} ( R^A \oplus M^A)\]
where $R^A \oplus M^A$ is the split square zero extension of $R^A =A^\vee$ with fiber the $R$-linear
mapping spectrum $M^A = \rm{map}_R(A, M)$. Thus, for $X=\pt$  the claim is 
immediate from Remark~\ref{Q0}. Observe that, by naturality of the Artin\textendash Schreier map, the map
\[(R\oplus M)\{X\} \to \prod_X(R\oplus M)\] 
induces a commutative diagram
\[\begin{tikzcd}
	{\Map_{\Mod_R}(A, (R\oplus M)\{X\})}  & {\prod_X \Map_{\Mod_R}(A, R\oplus M) }  \\
	{\Map_{\Mod_R}(A, (R\oplus M)\{X\})} & {\prod_X \Map_{\Mod_R}(A, R\oplus M)}
	\arrow[from=1-1, to=2-1]
	\arrow[from=1-1, to=1-2]
	\arrow["{Q_0-\id}", from=1-2, to=2-2]
	\arrow[from=2-1, to=2-2],
\end{tikzcd}\]
which we may further factor through
\[\begin{tikzcd}
	{\Map_{\Mod_R}(A, (R\oplus M)\{X\})} & {\Map_{\Mod_R}(A,(\prod_Xk)\otimes_k (R\oplus M))} \\
	{\Map_{\Mod_R}(A, (R\oplus M)\{X\})} & {\Map_{\Mod_R}(A, (\prod_X k) \otimes_K (R\oplus M))}
	\arrow[from=1-1, to=1-2]
	\arrow[from=1-1, to=2-1]
	\arrow["{Q_0-\id}", from=1-2, to=2-2]
	\arrow[from=2-1, to=2-2].
\end{tikzcd}\]
The right hand vertical map splits as desired by Lemma~\ref{Q0} since 
$(\prod_X k)\otimes_k(R\oplus M)$ is naturally a ring. Moreover, since $k$ is a field, 
we can choose a retraction $\prod_X k \to k\{X\}$, which exhibits the left hand vertical map
as a retract of the right hand one. Hence, the left hand map splits as well, proving the claim.
\end{proof}

\begin{proposition}\label{emfet}
    Let $k$ be a field of characteristic $p$, $R\in \CAlg(\Mod_k^{\rm{cn}})$ 
    and let $V \in \Mod_{k}^{\heartsuit}$ be a discrete $k$-vector space. The $R$-homology coalgebra
    $R[\Omega^\infty \Sigma^n V] \in \rm{cCAlg}(\rm{Mod}^{\cn}_k)$ is formally \'etale.
\end{proposition}
\begin{proof}
Denote by $e:R\to \Sigma^n R$ the zero section of the split square zero extension for some 
$n\geq 0$. Upon choosing
a basis for $V$, i.e.~picking a set $X$ and an isomorphism $k\{X\}\rar{\sim}V$, applying the 
(non-connective) Weil restriction $e_!$ to the pullback of Proposition~\ref{bbpb} yields a pullback
diagram in $\cCAlg(\Mod_R)$
\[\begin{tikzcd}
	{e_{!}e^{\pt}R[\Omega^\infty\Sigma^n V]} & {C_R( e_\pt e^\pt R \otimes \Sigma^n V)} \\
	R & {C_R(e_\pt e^\pt R \otimes \Sigma^nV)} 
	\arrow[from=1-1, to=1-2]
	\arrow[from=1-1, to=2-1]
	\arrow["{{e_!(F_{R\oplus M}-1)}}", from=1-2, to=2-2]
	\arrow[from=2-1, to=2-2].
 \end{tikzcd}\]
By Corollary~\ref{split} the square zero terms splits off on the right hand vertical map and we obtain a pullback
\[\begin{tikzcd}
	{e_{!}e^{\pt}R[\Omega^\infty\Sigma^nV]} & {C_R( R\otimes \Sigma^nV)} \\
	R & {C_R( R \otimes \Sigma^nV)}
	\arrow[from=1-1, to=1-2]
	\arrow[from=1-1, to=2-1]
	\arrow["{F_R-1}", from=1-2, to=2-2]
	\arrow[from=2-1, to=2-2].
\end{tikzcd}\]
However, by Proposition~\ref{bbpb} this pullback is given by $R[\Omega^\infty \Sigma^n V]$. Since
$e_!e^\pt R[\Sigma^n V]= R[\Sigma^n V]$ is connective, it follows from our observations 
in Construction~\ref{radj} that in fact
\[ R[\Sigma^n V] \simeq e_\pt e^\pt R[\Sigma^n V] \simeq e_! e^\pt R[\Sigma^n V].\]
In particular, the natural maps
\[ R[\Omega^\infty\Sigma^n V] \rar{\varepsilon} e_\pt e^\pt R[\Omega^\infty\Sigma^n V]
= \Omega^\infty_{R[\Omega^\infty\Sigma^n V]}M \rar{\pi} R[\Omega^\infty\Sigma^n V] \]
are equivalences and so $R[\Omega^\infty \Sigma^n V]$ is formally \'etale.
\end{proof}

We now expand from Eilenberg\textendash MacLane spaces to arbitrary connected spaces. 
Let us recall the following notions in unstable homotopy theory.

\begin{definition}
 A space $X$ is called \textit{$p$-complete} if it is local with respect to the functor
$\F_p[\blank]:\cl{S}\to \Mod_{\F_p}$. Denote by $\cl{S}_p\subseteq \cl{S}$ the full 
subcategory category of $p$-complete spaces. Moreover, call a space \textit{nilpotent}
if is connected and $\pi_1$ is a nilpotent group which acts nilpotently on the higher homotopy
groups. Write $\cl{S}^{\rm{nilp}}$ for the category of complete nilpotent spaces and
$\cl{S}_p^{\rm{nilp}}= \cl{S}_p \cap \cl{S}^\rm{nilp}$ for the category of $p$-complete
nilpotent spaces.
\end{definition}

\begin{lemma}\label{genem}
    Let $X$ be a connected space and $R$ a connective $\F_p$-algebra.
    Then $R[\Omega^\infty \F_p\{X\}] \in \cCAlg(\Mod_R^{\cn})$ is given by a limit
    $\lim R[X_i]$ where each $X_i$ is a finite product of Eilenberg\textendash MacLane spaces.
\end{lemma}
\begin{proof}
Since any $\F_p$-module is free, we can write 
\[
Y:=\Omega^\infty \F_p\{X\} \simeq \prod_{i\geq 0} \Omega^\infty \Sigma^i V_i
\]
where $V_i = \pi_i\F_p\{X\} \in \Mod_{\F_p}^{\heartsuit}$. Setting
\[
Y_n = \tau_{\leq n}Y = \prod_{i\leq n} \Omega^\infty \Sigma^i V_i 
\]
we have $Y \simeq \flim Y_n$. Since $R[\blank]$ commutes with Postnikov towers we get an equivalence
\[
cR[Y] \simeq \flim R[Y_n] \in \cCAlg(\rm{Pro}(\Mod_R)).
\]
moreover, since $\tau_{\leq k } R[Y] \simeq \tau_\leq R[Y_k]$ this diagram is pro-constant up 
to pro-truncation and hence we get 
\[
R[Y] \simeq M^{cA}(\flim R[Y_n]) = \flim R[Y_n] \in \cCAlg(\Mod_R)
\]
as claimed.
\end{proof}

\begin{lemma}\label{emhom}
    Let $X$ be a space and $R$ an $\bb{E}_\infty$-ring. There exists a cofiltered diagram
    of nilpotent spaces $\{X_{\alpha}\}$, such that the natural map
    \[ R[X] \to \flim_{\alpha} R[X_{\alpha}]\]
    is an equivalence in $\cCAlg(\Mod_R)$.
\end{lemma}
\begin{proof}
   Since finite limits of nilpotent spaces are nilpotent,~\cite[][Remark 3.1.7]{dag8}
   implies that the inclusion $\iota: \cl{S}^{\rm{nilp}} \to \cl{S}$ induces an adjunction
   \[\begin{tikzcd}
	{(\widehat{\blank}): \rm{Pro}(\cl{S})} & {\rm{Pro}(\cl{S}^{\rm{nilp}}):{\iota}}
	\arrow[""{name=0, anchor=center, inner sep=0}, shift left=2, from=1-2, to=1-1]
	\arrow[""{name=1, anchor=center, inner sep=0}, shift left=2, from=1-1, to=1-2]
	\arrow["\dashv"{anchor=center, rotate=-90}, draw=none, from=1, to=0].
\end{tikzcd}\]
  The left adjoint is given by taking a space $X$ to $ \widehat{X}=\flim_{X_{\alpha} \to X} X_{\alpha}$
  where the cofiltered limit runs over all maps $X_{\alpha} \to X$ where $X_{\alpha}$ is nilpotent. 
 Moreover, the homology $R[\blank]\colon \cl{S}\to \Mod_R$ induces a functor 
\[ R[\blank]\colon \rm{Pro}(\cl{S})\to \rm{Pro}(\Mod_R) \quad \flim_i X_i \mapsto \flim_i R[X_i].\]
  Since the infinite loop space of any spectrum is nilpotent, applying $R[\blank]$ to the reflection
  $X\to \flim_\alpha X_{\alpha}$ into $\rm{Pro}(\cl{S}^{\rm{nilp}})$ induces an equivalence
  \[ cR[X] \to \flim_{\alpha} R[X_\alpha]\]
  in $\rm{Pro}(\Mod_R)$. Since the forgetful functor 
  $\cCAlg(\rm{Pro}(\Mod_R)) \to \rm{Pro}(\Mod_R)$ commutes with cofiltered limits, the right 
  hand limit  can also be computed in $\cCAlg(\rm{Pro}(\Mod_R))$. Since the coalgebraic 
  materialization of constant objects exists and commutes with limits, we obtain that
  \[ R[X] \simeq M^{\rm{cA}}(cR[X]) \simeq M^{\rm{cA}}(\flim_\alpha R[X_{\alpha}])
  \simeq \flim_{\alpha} R[X_\alpha ],\]
 as claimed. 
\end{proof}

\begin{theorem}\label{sepet}
    Let $X$ be a connected space and $R$ a connective $\F_p$-algebra.
    Then the $R$-homology $R[X]\in \cCAlg(\Mod_R^{\rm{cn}})$ is formally \'etale. 
\end{theorem}
\begin{proof}
    Since $\Omega^\infty N$ is $p$-complete for any $N \in \Mod_{\F_p}$, the natural map
    \[ R[X] \to R[X^\wedge_p]\]
    is an equivalence, so we may assume that $X$ is $p$-complete. Let $M \in \Mod_R^{\rm{cn}}$ and
    denote by $e:R \to R\oplus M$ be the 0-section.
    By Lemma~\ref{emhom} we again have that
    \[ e^\pt R[X] = \flim e^\pt R[X_\alpha] \]
    where each $X_\alpha$ is nilpotent. Thus, since Weil-restriction again commutes with limits 
    we can assume that $X$ is nilpotent. Write 
    \[
    X^n := (\Omega^\infty \F \{ \blank\})^{\circ n} \circ X,
    \]
    which defines a co-augmented simplicial diagram $X^\bullet \to X$, which admits an additional
    degeneracy by choosing base-points. Then, since $X$ is nilpotent and $p$-complete we have
    by~\cite[][Proposition VI.6.2]{bk} that $X\simeq \lim X^n$. Moreover, the additional degeneracy
    tells us that this is a universal limit diagram, and hence
    \[ e^\pt R[X] \simeq \flim e^\pt R[X^\bullet] \in \cCAlg(\Mod_R).\]
    So we can assume that $X$ is of the form $\Omega^\infty \F_p\{Y\}$. Now, by Lemma~\ref{genem}
    we see that $e^\pt R[X] \simeq \lim_i e^\pt R[X_i]$ where each $X_i$ is a finite 
    product of Eilenberg\textendash MacLane spaces. Hence, the claim follows from Proposition~\ref{emfet}
    since $R[\blank]$ takes products of spaces to products of coalgebras.
\end{proof}

Thus, we have answered our initial question about recovering spherical chains from characteristic
$p$-chains, which we may summarize as follows.

\begin{corollary}
  Let $X\in \cl{S}$ be either connected or compact. The space of lifts of
  $k[X] \in \cCAlg(\Mod_k)$ to a coalgebra in $p$-complete, connective $\bb{W}(k)$-modules 
  is contractible and the unique poiint is given by $\bb{W}(k)[X]^\wedge$. 
  Moreover, for any $A \in \rm{cCAlg}(\Mod_{\bb{W}(k)}^{\wedge,\rm{cn}})$
  the base change along the map $\bb{W}(k)\to k$ induces an equivalence
  \[\rm{Map}_{\rm{cCAlg}(\rm{Mod}_{\bb{W}(k)}^{\wedge})}(A, \bb{W}(k)[X]^\wedge_p)
    \rar{\sim} \rm{Map}_{\rm{cCAlg}(\rm{Mod}_{k})}(A\otimes k, k[X]). \]
\end{corollary}
\begin{proof}
    Combine Theorem~\ref{sepet} with Theorem~\ref{wittsp}.
\end{proof}

Finally, we observe that this also yields a spherical and coalgebraic version
of Mandells embedding in~\cite{mandell}.

\begin{corollary}
 Let $k$ be a perfect, separably closed field of characteristic $p$ with spherical Witt vectors
 $\W(k)$. The $p$-complete $\bb{W}(k)$-homology functor
 \[ \cl{S}_p^{\rm{nilp}} \to \rm{cCAlg}(\Mod^{\wedge}_{\bb{W}(k)})
 \quad X \mapsto \bb{W}(k)[X]^\wedge_p\]
 is fully faithful. In particular, for any nilpotent space $X$ we have a natural equivalence
 \[ X^\wedge_p\simeq \rm{Map}_{\rm{cCAlg}(\Mod_{\bb{W}(k)}^\wedge})(\bb{W}(k), \bb{W}(k)[X]^\wedge_p).\]
\end{corollary}
\begin{proof}
By~\cite[][Theorem 1.2.]{bb} the functor
\[ k[\blank]: \cl{S}_p^{\rm{nilp}} \to \cCAlg(\Mod_{k})\]
is fully faithful. By Theorem~\ref{sepet} it factors through $\cCAlg(\Mod_k^{\rm{cn}})^{\fet}$ and
hence by Theorem~\ref{wittsp} we can compose with $\cl{W}$ and get the claim.
\end{proof}
\printbibliography
\end{document}